\documentclass[12pt,letterpaper]{article}

\usepackage{amssymb,amsfonts,amscd,amsthm}
\usepackage[all,arc]{xy}
\usepackage{enumerate, bbm}
\usepackage{mathrsfs}
\usepackage{tikz}
\usepackage[left=0.9in,top=0.9in,right=0.9in,bottom=0.9in]{geometry}
\usepackage{mathtools}
\usepackage{hyperref}
\usepackage{color}
\usepackage{graphicx}
\usepackage{cleveref}
\usepackage{xfrac}
\usepackage{xcolor}
\usepackage{soul}
\usepackage{svg}
\usepackage{float}
\usepackage{bookmark}
	
\usepackage{enumitem} 
\usepackage[font={small,it},labelfont=bf]{caption}
\graphicspath{ {TexImages/} }

\newcommand{\Hc}{\mathcal{H}}

\newtheorem{thm}{Theorem}[section]
\newtheorem{cor}[thm]{Corollary}
\newtheorem{prop}[thm]{Proposition}
\newtheorem{lem}[thm]{Lemma}
\newtheorem{claim}[thm]{Claim}

\newtheorem{conj}[thm]{Conjecture}

\theoremstyle{definition}
\newtheorem{defn}{Definition}
\crefname{defn}{definition}{definitions}
\crefname{claim}{claim}{claims}

\hypersetup{
	colorlinks,
	citecolor=blue,
	filecolor=blue,
	linkcolor=blue,
	urlcolor=blue,
	linktocpage
}

\setenumerate[2]{label*=\arabic*.} 
\setlist[enumerate]{itemsep=2ex, topsep=2ex} 
\setlist[itemize]{itemsep=2ex, topsep=2ex}

\def\A{\mathcal{A}}
\def\I{\mathcal{I}}

\def\S{\mathcal{S}}
\def\C{\mathcal{C}}

\newcommand{\cH}{\mathcal{H}}

\newcommand{\ex}{\mathrm{ex}}

\newcommand{\R}{\mathbb{R}}
\newcommand{\N}{\mathbb{N}}
\newcommand{\NN}{\mathbb{N}}

\renewcommand{\P}{\mathcal{P}}
\newcommand{\al}{a}
\newcommand{\be}{b}
\newcommand{\gam}{\gamma}

\newcommand{\sig}{\sigma}
\newcommand{\ep}{\varepsilon}

\newcommand{\Om}{\Omega}

\newcommand{\del}{\delta}
\newcommand{\Del}{\Delta}

\renewcommand{\l}{\left}
\renewcommand{\r}{\right}

\newcommand{\half}{\frac{1}{2}}
\newcommand{\quart}{\frac{1}{4}}

\newcommand{\sm}{\setminus}
\newcommand{\sub}{\subseteq}

\renewcommand{\c}[1]{\mathcal{#1}}

\newcommand{\tr}[1]{\textrm{#1}}

\renewcommand{\SS}[1]{\textcolor{red}{#1}}

\newcommand{\ve}{\varepsilon}

\title{Random Tur\'an Problems for Graphs with a Vertex Complete to One Part.}
\author{Sean Longbrake \footnote{Dept.\ of Mathematics, Emory University, sean.longbrake@emory.edu }\and Sam Spiro\footnote{Dept.\ of Mathematics and Statistics, Georgia State University, sspiro@gsu.edu}}
\date{\today}

\begin{document}

\maketitle


\begin{abstract}
    Given a graph $F$, the random Tur\'an problem asks to determine the maximum number of edges in an $F$-free subgraph of $G_{n,p}$.  Prior to this work, the only bipartite graphs $F$ with known tight bounds included certain classes of complete bipartite graphs and theta graphs.  We greatly expand upon these examples by proving tight bounds for a number of bipartite graphs which have a vertex complete to one part.  We also prove new general upper bounds for this problem which in many cases do significantly better than the only previous known general upper bound due to Jiang and Longbrake.  Our proofs utilize dependent random choice together with the recent technique of balanced vertex supersaturation in conjunction with hypergraph containers.
    
\end{abstract}

\section{Introduction}
This paper centers around probabilistic analogs of the Tur\'an problem.  Recall that the \textit{Tur\'an number} $\ex(n,F)$ of a graph $F$ is defined to be the maximum number of edges that an $n$-vertex $F$-free graph can have.  To define our random variant, we let $G_{n,p}$ denote the random graph on $n$ vertices obtained by including each possible edge independently and with probability $p$.  We then define the \textit{random Tur\'an number} $\ex(G_{n,p},F)$ to be the maximum number of edges in an $F$-free subgraph of $G_{n,p}$.  Note that when $p=1$ we have $\ex(G_{n,1},F)=\ex(n,F)$, so the random Tur\'an number can be viewed as a probabilistic analog of the classical Tur\'an number.

The asymptotics for $\ex(G_{n,p},F)$ are essentially known if $F$ is not bipartite due to independent breakthrough work of Conlon and Gowers \cite{conlon2016combinatorial} and of Schacht \cite{schacht2016extremal}.  Because of this, we focus primarily on the degenerate case when $F$ is bipartite where much less is known.

The bipartite random Tur\'an problem was originally studied in the case when $F$ was an even cycle.  Some initial results in this direction were given by Haxell, Kohayakawa, and \L uczak \cite{haxell1995turan} and Kohayakawa, Kreuter, and Steger \cite{kohayakawa1998extremal}, with a major breakthrough coming from work of Morris and Saxton~\cite{morris2016number} who proved essentially tight bounds for this problem assuming some well known conjectures on the maximum size of graphs with large girth.  In addition to this, \cite{morris2016number} went on to essentially solve the random Tur\'an problem for complete bipartite graphs $K_{r,t}$ in the following sense.  Here we say that a sequence of events $A_n$ holds \textit{asymptotically almost surely} or \textit{a.a.s.}\ for short if $\Pr[A_n]\to 1$ as $n\to \infty$, and we write $f(n)\ll g(n)$ if $f(n)/g(n)\to 0$ as $n\to \infty$.  
\begin{thm}[\cite{morris2016number}]\label{thm:Kst}
    If $\ex(n,K_{r,t})=\Theta(n^{2-1/r})$, then a.a.s.
    \[\ex(G_{n,p},K_{r,t})=\begin{cases}
        \Theta(p^{1-1/r}n^{2-1/r}) &  n^{-\frac{r-1}{rt-1}}(\log n)^{O(1)} \le p \\ 
        n^{2-\frac{r+t-2}{rt-1}}(\log n)^{\Theta(1)} & n^{-\frac{r+t-2}{rt-1}}\ll p\le n^{-\frac{r-1}{rt-1}}(\log n)^{O(1)},\\ 
        (1+o(1))p{n\choose 2} & n^{-2}\ll p \ll n^{-\frac{r+t-2}{rt-1}}.
    \end{cases}\]
\end{thm}
We note that $\ex(n,K_{r,t})$ is known to equal $\Theta(n^{2-1/r})$ whenever $t$ is sufficiently large in terms of $r$, so this gives essentially tight bounds for $\ex(G_{n,p},K_{r,t})$ whenever $t$ is large.

Since the breakthrough work of \cite{morris2016number} which appeared over a decade ago, the only additional tight examples that have been proven has come from recent work of McKinley and Spiro~\cite{mckinley2023random} who solved the problem for certain classes of theta graphs.  While no other tight bounds are known for specific graphs, there are a few general bounds known for the random Tur\'an problem.  For example, Spiro~\cite{spiro2024random} proved effective lower bounds on $\ex(G_{n,p},F)$ whenever $F$ is a power of a balanced rooted tree.  For upper bounds, the most general result is the following due to 
Jiang and Longbrake~\cite[Theorem 2.5]{jiang2022balanced}.

\begin{thm}[(Informal) \cite{jiang2022balanced}]\label{thm:JiangLongbrake}
    If $F$ is a graph with $\ex(n,F)=\Theta(n^{2-\alpha})$, then there exists some $\ep>0$ such that a.a.s.\ $\ex(G_{n,p},F)=O(p^{1-\ep}n^{2-\alpha})$ for all $p$ sufficiently large.  Moreover, if $F$ satisfies certain supersaturation conditions, then for all $p$ sufficiently large we have a.a.s.\ 
    \[\ex(G_{n,p},F)=O(p^{1-m_2^*(F)\alpha} n^{2-\alpha}),\]
    where
    \[m_2^*(F):=\max_{F'\subsetneq F,\ v(F')\ge 3} \frac{e(F')-1}{v(F')-2}.\]
\end{thm}
The bounds of \Cref{thm:JiangLongbrake} give a much simpler proof of the results of Morris and Saxton~\cite{morris2016number} for even cycles, though it is conjectured that the bounds of \Cref{thm:JiangLongbrake} are not tight whenever $F$ contains at least two cycles.

The results stated above capture every known result concerning the random Tur\'an problem for bipartite graphs, though we note that more can be said regarding the analogous problem of random Tur\'an numbers for degenerate hypergraphs \cite{jiang2024number, mubayi2023random, nie2023random, nie2023tur, nie202X, nie2023sidorenko, spiro2021relative}.  Despite the relative lack of knowledge for bipartite random Tur\'an numbers, there exists a recent conjecture of McKinley and Spiro predicting how $\ex(G_{n,p},F)$ should behave for arbitrary bipartite graphs $F$.  For this we introduce the following definition.

\begin{defn}
    Given a graph $F$ on at least 3 vertices, we define the \textit{2-density} of $F$ by
    \[m_2(F)=\max_{F' \sub F: v(F')\ge 3} \frac{e(F')-1}{v(F')-2},\]
    and say that $F$ is \textit{2-balanced} if $m_2(F)=\frac{e(F)-1}{v(F)-2}$.
\end{defn}
Note that $m_2(F)$ is defined in the same way as $m_2^*(F)$ in \Cref{thm:JiangLongbrake} except that we allow for $F'=F$.

\begin{conj}[\cite{mckinley2023random}]\label{conj:MS}
    If $F$ is a graph with $\ex(n,F)=\Theta(n^{2-\alpha})$ for some $\alpha\in (0,1)$, then a.a.s.\
    \[\ex(G_{n,p},F)=  \begin{cases}\max\{\Theta(p^{1-\alpha}n^{2-\alpha}), n^{2-1/m_2(F)}(\log n)^{\Theta(1)}  \} & p\gg n^{-1/m_2(F)},\\ (1+o(1))p {n\choose 2} & n^{-2}\ll p\ll n^{-1/m_2(F)}.\end{cases}\]
\end{conj}
In particular, this conjecture predicts that $\ex(G_{n,p},F)$ should always have three ranges of behavior (i.e.\ it should roughly equal either $p^{1-\alpha}n^{2-\alpha},\ n^{2-1/m_2(F)}$, or $p{n\choose 2}$ depending on the value of $p$), and moreover, it predicts that one of these ranges will be a ``flat middle range,'' i.e.\  a range where $\ex(G_{n,p},F)=n^{2-1/m_2(F)}(\log n)^{\Theta(1)}$ is essentially independent of $p$ for a sizable range of $p$ close to $n^{-1/m_2(F)}$. 

Surprisingly, \Cref{conj:MS} has a close connection to Sidorenko's conjecture.  Very informally, a graph $F$ is \textit{Sidorenko} if every dense graph $G$ has about as many copies of $F$ as we would expect in the random graph with the same density of $G$, with Sidorenko\cite{sidorenko1986extremal, sidorenko1991inequalities} famously conjecturing that every bipartite graph is Sidorenko.  

In recent work, Nie and Spiro~\cite{nie2023sidorenko} showed that if $F$ is a 2-balanced bipartite graph which satisfies \Cref{conj:MS}, then $F$ must be Sidorenko. Given this, one can only hope to prove the tight bounds of \Cref{conj:MS} for (2-balanced) graphs $F$ in the case where it is already known that $F$ is Sidorenko.  While quite a lot of work has gone into proving various special cases of Sidorenko's Conjecture \cite{conlon2010approximate, conlon2018some, conlon2017finite, coregliano2021biregularity, fox2017local, hatami2010graph, kim2016two, li2011logarithimic, lovasz2011subgraph, szegedy2014information}, only a very limited set of bipartite graphs are known to be Sidorenko, and hence there are very few graphs $F$ for which we can attempt to prove the bounds of \Cref{conj:MS} for.

\section{Main Results}
Because we can only hope to prove tight bounds for $\ex(G_{n,p},F)$ when $F$ is known to be Sidorenko, we will focus our attention on certain classes of Sidorenko graphs.  Notably, an important result of Conlon, Fox, and Sudakov \cite{conlon2010approximate} proves that every bipartite graph which has a vertex complete to one part is Sidorenko, and it is graphs of this form that we will address in this paper.  We prove our results using a variety of techniques, such as dependent random choice, hypergraph containers, and the recently developed method of vertex balanced supersaturation.


Our most general results are somewhat technical to state, and as such, we postpone them to \Cref{sub:technical}.  Until then, we highlight some of the main applications of our results, including an improved bound compared to the general upper bound \Cref{thm:JiangLongbrake} for certain classes of graphs and values of $p$, as well as a new infinite family of graphs for which we can obtain tight bounds for all values of $p$.

\subsection{General Upper Bounds}
F\"uredi~\cite{furedi1991turan} proved that $\ex(n,F)=O(n^{2-1/r})$ whenever $F$ has a bipartition $S\cup T$ such that every vertex in $T$ has degree at most $r$, and this result was later given an elegant proof by Alon, Krivelevich, and Sudakov~\cite{alon2003turan} using dependent random choice.  The dependent random choice proof of this result in fact further implies that this same bound continues to hold even if $T$ contains a small number of vertices of degree larger that $r$.  Our first main result is a probabilistic analogue of this bound in the case where we allow up to 1 vertex to have degree larger than $r$.  This gives the only other general upper bound for the bipartite random Tur\'an problem beyond that of \Cref{thm:JiangLongbrake}, with this new bound having the additional advantage of being provably tight for a large number of graphs.

\begin{thm}\label{thm:maxDegree}
    Let $F$ be a bipartite graph and $r,\Del \ge 2$ integers such that $F$ has a bipartition $S\cup T$ and a vertex $v^*\in T$ satisfying that every vertex in $T\sm v^*$ has degree at most $r$ and every $u\in S$ has $|N_F(u)\cup \{v^*\}|\le \Del$.  Then there exists some $C$ such that a.a.s.
    \[\ex(G_{n,p},F)=O(p^{1-1/r} n^{2-1/r})\ \tr{ for all }\ p\ge n^{-\frac{r-1}{r\Del-1}}(\log n)^C.\]
    If moreover $F$ contains a subgraph isomorphic to a complete bipartite graph $K_{r,t}$ satisfying $\ex(n,K_{r,t})=\Theta(n^{2-1/r})$, then a.a.s.
    \[\ex(G_{n,p},F)=\Theta(p^{1-1/r} n^{2-1/r})\ \tr{ for all }\ p\ge n^{-\frac{r-1}{r\Del-1}}(\log n)^C.\]
\end{thm}

This result exactly recovers the (non-trivial) upper bounds of \Cref{thm:Kst} for complete bipartite graphs by taking $F=K_{r,\Del}$.  It also does much better than the general upper bound \Cref{thm:JiangLongbrake} whenever $\ex(n,F)=\Theta(n^{2-1/r})$.  Indeed, in this case one can check that \Cref{thm:JiangLongbrake} always gives an upper bound of the form $O(p^{1-1/r-\ep}n^{2-1/r})$ for some $\ep>0$ whenever $F$ contains at least two cycles, showing that our bound on $\ex(G_{n,p},F)$ is qualitatively better in this case. If moreover the graph $F$ is assumed to contain a copy of $K_{r,r}$ and at least one other edge (which is roughly conjectured to be the only way a graph as in \Cref{thm:maxDegree} can satisfy $\ex(n,F)=\Theta(n^{2-1/r})$ \cite[Conjecture 1.2]{conlon2021extremal}), then one can check that the bound of \Cref{thm:JiangLongbrake} is at most $O(p^{\frac{1}{2}-\frac{1}{2r}}n^{2-1/r})$ which is substantially weaker than the bound $O(p^{1-1/r} n^{2-1/r})$ coming from our theorem.

Our methods can also be used to give effective bounds for small values of $p$ provided that $F$ has a vertex complete to one side and satisfies a certain ``balanced'' condition.  
\begin{thm}\label{thm:upperSmallp}
    Let $F$ be a graph with $e(F)\ge 2$ which has a bipartition $S\cup T$ such that there exists a vertex $v^*\in T$ which is adjacent to every vertex in $S$, and such that any set $\nu\sub V(F)$ with $|\nu|\ge 3$ and $\frac{e(F[\nu])-1}{|\nu|-2}=m_2(F)$ has $S\sub \nu$.  Then a.a.s.\ 
    \[\ex(G_{n,p},F)=n^{2-1/m_2(F)}(\log n)^{\Theta(1)} \ \textrm{for all }\  n^{\frac{-1}{m_2(F)}}\ll p \le n^{\frac{1}{e(F)^2}-\frac{1}{m_2(F)}}.\]
\end{thm}

In particular, this result holds whenever $F$ both has a vertex complete to one side and is \textit{strictly 2-balanced}, meaning that $\frac{e(F[\nu])-1}{|\nu|-2}=m_2(F)$ is satisfied only for $\nu=V(F)$.

\subsection{Tight Examples}
We next provide new infinite families of graphs $F$ for which we can prove tight bounds for $\ex(G_{n,p},F)$ at all values of $p$.  
The first set of graphs that we prove these bounds for will be of the following form.

\begin{defn}
    Given a multigraph $M$ without loops, we define the graph $F_M$ by subdividing each edge of $M$ and then adding a new vertex $v^*$ which is made adjacent to every vertex in $V(M)$.    
\end{defn}
\begin{figure}
\[\begin{tikzpicture}[scale=2,
    vertex/.style={circle, fill=black, inner sep=1.5pt}
]

\node[vertex] (a) at (90:1) {};
\node[vertex] (b) at (210:1) {};
\node[vertex] (c) at (330:1) {};

\node[vertex] (d) at (0,0) {};

\node[vertex] at (0,-.5) {};
\node[vertex] at (.57,.37) {};
\node[vertex] at (.43,.24) {};
\node[vertex] at (-.57,.37) {};
\node[vertex] at (-.43,.24) {};
\node[vertex] at (.29,.13) {};

\draw (a) -- (b);
\draw (b) -- (c);
\draw (a) -- (c);

\draw (a) -- (d);
\draw (b) -- (d);
\draw (c) -- (d);

\draw[bend left=20] (a) to (c);
\draw[bend right=20] (a) to (c);

\draw[bend right=20] (a) to (b);

\end{tikzpicture}\]
\caption{A depiction of the graph $F_M$ when $M$ is a triangle with edge multiplicities 1, 2, and 3.  By \Cref{thm:multigraph}, this graph is a new tight example for the random Tur\'an Problem.}
\end{figure}
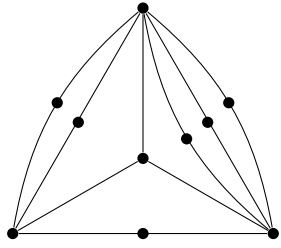

We give tight bounds for $F_M$ (which also match the bounds predicted by \Cref{conj:MS}) whenever $M$ satisfies a certain balanced condition.
\begin{thm}\label{thm:multigraph}
    If $M$ is a multigraph with $e(M)\ge 1$ such that
    \[\max_{\mu\sub V(M),\ |\mu|\ge 2}\frac{e(M[\mu])}{|\mu|-1}=\frac{e(M)}{v(M)-1},\]
    then a.a.s.
    \[\ex(G_{n,p},F_M)=\begin{cases}
        \Theta(p^{1/2}n^{3/2}) &  n^{-\frac{v(M)-1}{v(M)+2e(M)-1}}(\log n)^{O(1)}\le p,\\ 
        n^{2-\frac{v(M)+e(M)-1}{v(M)+2e(M)-1}}(\log n)^{\Theta(1)} & n^{-\frac{v(M)+e(M)-1}{v(M)+2e(M)-1}}\ll p\le n^{-\frac{v(M)-1}{v(M)+2e(M)-1}}(\log n)^{O(1)},\\ 
        (1+o(1))p{n\choose 2} & n^{-2}\ll p\ll n^{-\frac{v(M)+e(M)-1}{v(M)+2e(M)-1}}.
    \end{cases}\]

\end{thm}
For example, if $M$ is a triangle with edge multiplicities $a\ge b\ge c$, then \Cref{thm:multigraph} applies if and only if $a\le b+c$.  \Cref{thm:multigraph} is the simplest case of the following more general result giving tight bounds for certain graphs $F$ satisfying some balanced conditions. 

\begin{thm}\label{thm:generalTightExamples}
    Let $F$ be a graph which has a bipartition $S\cup T$ and a vertex $v^*\in T$ such that $v^*$ is adjacent to all of $S$ and such that every $v\in T\sm \{v^*\}$ has degree exactly $r\ge 2$.  For each $\mu\sub S$, let $N(\mu)\sub T$ denote the set of vertices of $T$ that are adjacent to at least one vertex of $\mu$.  
    
    If $F$ contains a subgraph isomorphic to some $K_{r,t}$ with $\ex(n,K_{r,t})=\Theta(n^{2-1/r})$, if $F$ is 2-balanced, and if $F$ satisfies
    \[\max_{\mu\sub S,\ |\mu|\ge 2} \frac{e(F[\mu\cup N(\mu)])-|N(\mu)|}{|\mu|-1}=\frac{e(F)-|T|}{|S|-1},\]
    then a.a.s.
    \[\ex(G_{n,p},F)=\begin{cases}
        \Theta(p^{1-1/r}n^{2-1/r}) &  n^{-\frac{|S|-1}{|S|-1+r(|T|-1)}}(\log n)^{O(1)}\le p,\\ 
        n^{2-\frac{|S|+|T|-2}{|S|-1+r(|T|-1)}}(\log n)^{\Theta(1)} & n^{-\frac{|S|+|T|-2}{|S|-1+r(|T|-1)}}\ll p\le n^{-\frac{|S|-1}{|S|-1+r(|T|-1)}}(\log n)^{O(1)},\\ 
        (1+o(1))p{n\choose 2} & n^{-2}\ll p\ll n^{-\frac{|S|+|T|-2}{|S|-1+r(|T|-1)}}.
    \end{cases}\]
\end{thm}

For example, consider the graph $F_{r,s,t}$ defined by starting with an arbitrary set of $s$ vertices $S$ and a vertex $v^*$ adjacent to all of $S$, and then for each $R\sub S$ of size $r$, we introduce new vertices $v_{1,R},\ldots,v_{t,R}$ which are adjacent to all of $R$ and no other vertices of $S$. One can check that graphs of this form satisfy the conditions of \Cref{thm:generalTightExamples}, giving yet another set of new tight examples

We once again emphasize that prior to these results, the only bipartite graphs for which non-trivial tight bounds for the random Tur\'an problem  were known was limited only to certain complete bipartite graphs, cycles, and theta graphs \cite{mckinley2023random,morris2016number}.  As such, our results Theorems \ref{thm:multigraph} and \ref{thm:generalTightExamples} give a much richer set of known tight examples.

\subsection{The Most General Theorems}\label{sub:technical}

We now state our most general theorems which will imply all of the upper bounds of the previous subsections.  Our general theorems will apply to the following types of graphs.
\begin{defn}
    We say that a bipartite graph $F$ is \textit{$r$-semi-bounded} with respect to a triple $(S,T,v^*)$ if $S\cup T$ is a bipartition of $F$, if $v^*\in T$ is a vertex which is adjacent to every vertex in $S$, and if every $v\in T\sm \{v^*\}$ has $\deg_F(v)\le r$.  We say that $F$ is $r$-semi-bounded if it is $r$-semi-bounded with respect to some triple.
\end{defn}
Here the ``semi'' part of the name ``$r$-semi-bounded'' is meant to refer both to the boundedness condition only applying to one of the parts of $F$, as well as to the fact that only $|T|-1$ of the vertices of $T$ have bounded degree.  We define a few parameters associated with $r$-semi-bounded graphs to state our main results.

\begin{defn}
    Given a bipartite graph $F$ which is $r$-semi-bounded with respect to a triple $(S,T,v^*)$, we define for each $\nu\sub V(F)$ the quantity
    \[e(\nu):=e(F[\nu]),\]
		and if $e(\nu)\ge 1$ we define 
		\[f(\nu):=|S\cap \nu|+\sum_{v\in T\cap \nu} \deg_F(v)-\max_{v\in T\cap \nu: \deg_{F[\nu]}(v)\ge 1} \deg_F(v).\]
\end{defn}

Note that having $e(\nu)\ge 1$ implies the maximum in the definition of $f(\nu)$ is over a non-empty set, meaning that $f(\nu)$ is well-defined in this case.  The strangely defined parameter $f(\nu)$ can be thought of as a weighted version of the more natural edge count $e(\nu)$.  In particular, we will formally show in \Cref{lem:fvse} that we always have $f(\nu)\ge e(\nu)$ with equality holding in some important cases.

The range for which we can bound $\ex(G_{n,p},F)$ for large $p$ will be controlled by the following parameter, where here $\del(F')$ denotes the minimum degree of the graph $F'$.	 
\begin{defn}
    Given an $r$-semi-bounded-graph $F$, let 
\[A_F=\{\nu: \del(F[\nu])\ge 1,\ f(\nu)-1< r(e(\nu)-1)\},\]
    and for $\nu\in A_F$ we define
    \[\al(\nu)=\frac{r(|\nu|-2)+1-f(\nu)}{r(e(\nu)-1)+1-f(\nu)}.\]and
    \[\al(F)=\min_{\nu\in A_F} \al(\nu).\]
\end{defn}
Technically the definition of $\al(F)$ depends not only on $F$, but also the choice of $r$ and its implicit triple $(S,T,v^*)$, but we will suppress these dependencies from $\al$ for ease of notation. 
 Observe that having $\nu\in A_F$ means that $\al(\nu)$ is well-defined (i.e.\ its denominator is not equal to $0$).  Later in \Cref{lem:alBound} we show that $A_F\ne \emptyset$ whenever $F$ contains a cycle, implying that $\al(F)$ is well-defined in this case, which will be the only case where we consider $\al(F)$.

The exact formulation of $\al(F)$ is rather opaque, but at a high-level it can be thought of as a variant of the 2-density of $F$.  More precisely, if $\nu$ is such that $f(\nu)=e(\nu)$, then we have $\al(\nu)=\frac{r(|\nu|-2)}{(r-1)(e(\nu)-1)}-\frac{1}{r-1}$, and hence $\al(\nu)$ is equal to a linear transformation of the usual (inverse) 2-density formula $\frac{|\nu|-2}{e(\nu)-1}$ in this case.

With these definitions in mind, we can state our first technical result giving effective upper bounds on $\ex(G_{n,p},F)$ whenever $p$ is sufficiently large relative to $\al(F)$.
\begin{thm}\label{thm:alpha}
    If $F$ is $r$-semi-bounded and contains a cycle, then there exists a constant $C>0$ such that for all $n^{-\al(F)}(\log n)^C\le p\le 1$ we have a.a.s.
    \[\ex(G_{n,p},F)=O(p^{1-1/r}n^{2-1/r}).\]
\end{thm}
We will show later in \Cref{prop:alBound} that $\al(F)>0$ whenever $F$ contains a cycle, meaning that \Cref{thm:alpha} always gives an effective upper bound for some non-empty interval of $p$ for every $r$-semi-bounded graph $F$.  We can also establish bounds at small values of $p$, for which we will need another parameter.

\begin{defn}
    Given an $r$-semi-bounded graph $F$, let 
    \[B_F=\{\nu: \del(F[\nu])\ge 1,\ f(\nu)>e(\nu)\}\]
    and for $\nu\in B_F$ define
    \[\be(\nu)=\frac{|\nu|-2-(e(\nu)-1)/m_2(F)}{f(\nu)-e(\nu)}\]
    and $\be(F)=\min_{\nu\in B_F} \be(\nu)$.
\end{defn}
Again, we will show in \Cref{lem:alBound} that $B_F\ne \emptyset$ whenever $F$ contains a cycle, implying that $\be(F)$ is well-defined in this case.
We will ultimately be able to conclude the following bounds in terms of this parameter.
\begin{thm}\label{thm:beta}
    If $F$ is $r$-semi-bounded and contains a cycle, then there exists some $C>0$ such that for all $n^{-1/m_2(F)}\le p\le \min\{n^{\be(F)-1/m_2(F)},n^{1/r-1/m_2(F)}\}$ we have a.a.s.
    \[\ex(G_{n,p},F)\le n^{2-1/m_2(F)}(\log n)^C.\]
\end{thm}

We are not aware of any graph $F$ as in \Cref{thm:beta} which has $\be(F)>1/r$, meaning it might be possible to drop the $n^{1/r-1/m_2(F)}$ term from the minimum, but we can not rule out this possibility.

While we noted how \Cref{thm:alpha} always gives a non-trivial bound for $r$-semi-bounded graphs $F$, this is not be the case for \Cref{thm:beta}. Indeed, one can check that $\be(F)=0$ precisely when $F$ fails to satisfy the balanced condition in \Cref{thm:upperSmallp}.
A simple example of such a graph with $\be(F)=0$ can be seen by taking $F$ to be a $K_{3,3}$ together with one additional vertex made adjacent to some $v^*,v'$ which lie on the same side of the bipartition, as in this case the set of six vertices $\nu$ which make up the $K_{3,3}$ satisfies $\nu\in B_F$ and $\be(\nu)=0$.



 \textbf{Organization}.  In \Cref{sec:preliminaries} we collect some auxiliary facts about our technical definitions such as $f(\nu)$ and $a(F)$.  In \Cref{sec:supersaturation} we prove our main balanced supersaturation results, and in \Cref{sec:randomTuranResults} we use this balanced supersaturation to prove our main results. We discuss some concluding remarks in \Cref{sec:concluding}.

\section{Preliminaries}\label{sec:preliminaries}
In this section we record all of the facts regarding the technical definitions made in \Cref{sub:technical} we need throughout the paper.  The reader may wish to skip or skim this section on a first read in order to get to the real heart of the paper.

We begin with the only lemma we need to prove our main balanced supersaturation result in \Cref{sec:supersaturation}.  In fact, to a large extent the main reason why we define 
\[f(\nu)=|S\cap \nu|+\sum_{v\in T\cap \nu} \deg_F(v)-\max_{v\in T\cap \nu: \deg_{F[\nu]}(v) \geq 1} \deg_F(v)\]
is so that the following lemma holds as stated.

\begin{lem}\label{lem:remove}
    Let $F$ be an $r$-semi-bounded graph with respect to some triple $(S,T,v^*)$ and let $\nu \sub V(F)$ be such that $e(\nu)\ge 1$.
    \begin{itemize}
        \item[(a)] If $F[\nu]$ is an edge, then $f(\nu)=1$.
        \item[(b)] If $u\in S\cap \nu$ satisfies $e(\nu\sm \{u\})\ge 1$, then
        \[f(\nu)-f(\nu\sm \{u\})=1.\]
        \item[(c)] If there exists a vertex in $T \cap \nu\sm \{v^*\}$ which is not incident to every edge of $F[\nu]$, then there exists some vertex $v\in T \cap \nu \sm \{v^*\}$ satisfying
        \[f(\nu)-f(\nu\sm \{v\})=\deg_F(v).\]
    \end{itemize}
\end{lem}

The high-level intuition for the this lemma is that given an arbitrary $\nu$, we can iteratively remove vertices as in (b) and (c) from $\nu$ until we eventually arrive at a star, and this will in turn allow us to focus most of our analysis in \Cref{sec:supersaturation} on dealing with ``bad'' stars. 

\begin{proof}
    Part (a) is straightforward.  For (b), we have
    \begin{align*}f(\nu)-f(\nu\sm \{u\})=|S\cap \nu|-|S\cap \nu \sm \{u\}|=1,\end{align*}
    where we implicitly used $e(\nu \sm \{u\})\ge 1$ to guarantee that $f(\nu\sm \{u\})$ is well-defined, and the last equality used $u\in S$.

    For (c), let $v\in \nu \cap T\sm \{v^*\}$ be an arbitrary vertex with no neighbors in $\nu$, and if no such vertex exists then let $v\in T\sm \{v^*\}$ be such that $\deg_F(v)$ is as small as possible.  We claim in either case that
    \[\max_{w\in T\cap\nu:\deg_{F[\nu]}(w)\ge 1} \deg_F(w)=\max_{w\in T\cap \nu\sm\{v\}:\deg_{F[\nu\sm\{v\}]}(w)\ge 1}\deg_F(w).\]
    This is immediate if $v$ has no neighbors in $\nu$, so we may assume that every vertex in $T\cap \nu$ has a neighbor in $\nu$.  This together with the hypothesis of (c) implies that the set $T\cap \nu \sm \{v\}$ is not empty.  The equality above then follows from the definition of $v$, as well as the trivial inequality $\deg_F(v)\le \deg_F(v^*)$ coming from the definition of $F$ being $r$-semi-bounded.   
    
    With the equality above in mind, we have
    \begin{align*}f(\nu)-f(\nu\sm \{v\})=\sum_{w\in T\cap \nu}\deg_F(w)-\sum_{w\in T\cap \nu\sm \{v\}} \deg_F(w)=\deg_F(v),\end{align*}
    proving the result.
\end{proof}

We next prove a result relating $f(\nu)$ with the more natural parameter $e(\nu)=e(F[\nu])$.

\begin{lem}\label{lem:fvse}
    Is $F$ is $r$-semi-bounded graph with respect to a triple $(S,T,v^*)$ and if $\nu \sub V(F)$ has $e(\nu)\ge 1$, then
    \[f(\nu)\ge e(\nu).\]
    Moreover, we have $f(\nu)=e(\nu)$ whenever $S\cup \{v^*\}\sub \nu$. 
\end{lem}
\begin{proof}
    Let $w\in T\cap \nu$ be such that $\deg_F(w)=\max_{v\in T\cap \nu:\deg_{F[\nu]}(v)\ge 1} \deg_F(v)$.  Then
    \begin{align}f(\nu)=|S\cap \nu|+\sum_{v\in T\cap \nu\sm \{w\}} \deg_F(v)\ge \deg_{F[\nu]}(w)+\sum_{v\in T\cap \nu\sm \{w\}} \deg_{F[\nu]}(v)=e(\nu),\label{eq:e vs f}\end{align}
    where the inequality used that $w$ can have at most $S\cap \nu$ neighbors in $F[\nu]$ and the trivial relation $\deg_F(v)\ge \deg_{F[\nu]}(v)$.  On the other hand, it is straightforward to check that if $S\sub \nu$ then $\deg_F(v)=\deg_{F[\nu]}(v)$ for each $v\in T\cap \nu \sm \{w\}$, and that if $v^*\in \nu$ then $|S\cap \nu|=\deg_{F[\nu]}(v^*)\le \deg_{F[\nu]}(w)$ by definition of $w$, so $\deg_{F[\nu]}(w)=|S\cap \nu|$, proving that \eqref{eq:e vs f} is an equality in this case.
\end{proof}

We will also need a few additional bounds on $f$.
\begin{lem}\label{lem:fBounds}
    If $F$ is $r$-semi-bounded with respect to some triple $(S,T,v^*)$, then for all $\nu\sub V(F)$ with $e(\nu)\ge 1$ we have
    \[f(\nu)\le e(F),\]
    as well as
    \[f(\nu)\le -(r-1)(|S\cap \nu|-1)+r(|\nu|-2)+1.\]
\end{lem}
\begin{proof}

    For the first inequality, if $v^*\in \nu$ then we have
    \[f(\nu)=|S\cap \nu|+\sum_{v\in T\cap \nu } \deg_F(v)-\deg_F(v^*)\le \sum_{v\in T\cap \nu } \deg_F(v)\le \sum_{v\in T} \deg_F(v)=e(F).\]
    On the other hand, if $v^*\notin \nu$ then
    \[f(\nu)\le \deg_F(v^*)+\sum_{v\in T\cap \nu }\deg_F(v)\le \sum_{v\in T} \deg_F(v)=e(F).\]
    
    We now turn to the second inequality.  Let $w$ be a vertex with $\deg_F(w)$ as large as possible amongst the vertices of $T\cap \nu$ with degree at least 1 in $F[\nu]$.  Because $F$ is $r$-semi-bounded, we have $\deg_F(v)\le r$ for all $v\in T\cap \nu\sm \{w\}$, and hence
    \begin{align*}f(\nu)&=|S\cap \nu|+\sum_{v\in T\cap \nu \sm \{w\}} \deg_F(v)\le |S\cap \nu|+r(|T\cap \nu|-1)\\ &=|S\cap \nu|+r(|\nu|-| S\cap \nu|-1)=-(r-1)(|S\cap \nu|-1)+r(|\nu|-2)+1.\end{align*}
\end{proof}

We now turn to some results around $A_F,B_F$, for which we recall
\[A_F=\{\nu: \del(F[\nu])\ge 1,\ f(\nu)-1< r(e(\nu)-1)\},\]
    and
    \[\al(F)=\min_{\nu\in A_F} \al(\nu)=\min_{\nu\in A_F}\frac{r(|\nu|-2)+1-f(\nu)}{r(e(\nu)-1)+1-f(\nu)},\]
    and
    \[B_F=\{\nu: \del(F[\nu])\ge 1,\ f(\nu)>e(\nu)\}.\]

\begin{lem}\label{lem:alBound}
    If $F$ is $r$-semi-bounded with respect to some $(S,T,v^*)$ and if $F$ contains a cycle, then the sets $A_F,B_F$ are non-empty (and hence $\al(F),\be(F)$ are well-defined).  Moreover, we have
    \[\al(F)\le 1.\]
\end{lem}

\begin{proof}
    Observe that if $F$ contains a cycle, then we must have $r\ge 2$ as otherwise $F$ would contain at most 1 vertex which has degree greater than 1, contradicting $F$ containing a cycle.  Similarly we must have $|S|\ge 2$.

    For the $A_F$ result, observe that the set $\nu=S\cup \{v^*\}$ has $e(\nu)-1=|S|-1$ and $f(\nu)-1=|S|-1$, which means $f(\nu)-1<r(e(\nu)-1)$ since $r,|S|\ge 2$, proving that $\nu\in A_F$.  Moreover, we have
    \[\al(F)\le \al(\nu)=\frac{r(|S|-1)+1-|S|}{r(|S|-1)+1-|S|}=1.\]

    We next consider $B_F$.  Because $F$ contains a cycle, there exist at least two vertices in $T$ with degree at least 2, and at least one of these vertices $v$ will not equal $v^*$.  Letting $u$ denote one of the neighbors of $v$, we see that $\nu=\{u,v,v^*\}$ satisfies \[f(\nu)=|S\cap \nu|+\deg_F(v)\ge 3>2=e(\nu),\] so $\nu\in B_F$, proving the result.
\end{proof}

In addition to the upper bound $\al(F)\le 1$ above, we will need a lower bound for $\al(F)$.

\begin{prop}\label{prop:alBound}
    If $F$ is $r$-semi-bounded with respect to $(S,T,v^*)$, if $F$ contains a cycle, and if every vertex in $S$ has degree at most $\Del$, then 
    \[\al(F) \ge \frac{r-1}{r\Del-1}.\]
\end{prop}
We note that the hypothesis of $F$ containing a cycle is needed only to guarantee that $\al(F)$ exists.
\begin{proof}
    Proving this is equivalent to showing that every $\nu\in A_F$ satisfies $\al(\nu) \ge \frac{r-1}{r\Del-1}$.  From now on we fix some arbitrary $\nu\in A_F$ and aim to show this inequality.  For this argument we make frequent use of the elementary fact that if $n,p,d,D$ are real numbers with $0<d<D$ and $n\le 0\le p$, then
    \begin{equation}
        \frac{p}{d}\ge \frac{p}{D}, \label{eq:dumb}
    \end{equation}
    and 
    \begin{equation}
        \frac{n}{d}\le \frac{n}{D}, \label{eq:dumb2}
    \end{equation}
    We will in particular start with an application of \eqref{eq:dumb} using $d:=r(e(\nu)-1)+1-f(\nu)$, which is positive by definition of $\nu\in A_F$.

    First consider the case that $e(\nu)\le |\nu|-1$.  We observe (irregardless of this case) that by \Cref{lem:fBounds}, the numerator of $\al(\nu)$ satisfies
    \[r(|\nu|-2)+1-f(\nu)\ge (r-1)(|S\cap \nu|-1)\ge 0,\]
    with this last inequality using that $\nu\in A_F$ implies $\del(F[\nu])\ge 1$, which in particular requires $S\cap \nu$ to be non-empty.  We can thus apply \eqref{eq:dumb} with $p$ the numerator of $\al(\nu)$ to conclude  in the case $e(\nu)\le |\nu|-1$ that
    \[\al(\nu)=\frac{r(|\nu|-2)+1-f(\nu)}{r(e(\nu)-1)+1-f(\nu)}\ge \frac{r(|\nu|-2)+1-f(\nu)}{r(|\nu|-2)+1-f(\nu)}=1\ge \frac{r-1}{r\Del-1},\]
    giving the  bound.

    We now consider the case $e(\nu)\ge |\nu|$.  Letting $d:=r(e(\nu)-|T\cap \nu|)+1-|S\cap \nu|$, we see by the assumption of $e(\nu)\ge |\nu|= |S\cap \nu|+|T\cap \nu|$ that $d>0$.  By \Cref{lem:fBounds} we have \[1-f(\nu)\ge (r-1)(|S\cap \nu|-1)-r(|\nu|-2)=r(|T\cap \nu|-1)+1-|S\cap \nu|,\] and hence $D:=r(e(\nu)-1)+1-f(\nu)\ge d$.   These observations together with an application of \eqref{eq:dumb2} gives
    \begin{align}\al(\nu)&=1+\frac{r(|\nu|-e(\nu)-1)}{r(e(\nu)-1)+1-f(\nu)}\nonumber \\ &\ge 1+\frac{r(|\nu|-e(\nu)-1)}{r(e(\nu)-|T\cap \nu|)+1-|S\cap \nu|} =\frac{(r-1)(|S\cap \nu|-1)}{r(e(\nu)-|T\cap \nu|)+1-|S\cap \nu|}.\label{eq:dumbApp}\end{align}
    Let $m=\max_{u\in S\cap \nu} \deg_{F[\nu]}(u)$.  Trivially we have $e(\nu)\le m|S\cap \nu|$ and $|T\cap \nu|\ge m$, which in total implies 
    \[d=r(e(\nu)-|T\cap \nu|)+1-|S\cap \nu|\le (rm-1)(|S\cap \nu|-1):=D'.\]
    We can thus apply \eqref{eq:dumb} to \eqref{eq:dumbApp} to conclude that
    
    \[\al(\nu)\ge \frac{(r-1)(|S\cap \nu|-1)}{(rm-1)(|S\cap \nu|-1)}=\frac{r-1}{rm-1}\ge \frac{r-1}{r\Del -1}.\]
    We conclude the result.
\end{proof}

Finally, we record a simplification of $f(\nu)$ in the following special case, the formulation of which is straightforward to verify.
\begin{lem}\label{lem:simple}
    If $F$ is $r$-semi-bounded with triple $(S,T,v^*)$ and if every $v\in T\sm \{v^*\}$ has degree exactly $r$, then every $\nu\sub V(F)$ with $e(\nu)\ge 1$ satisfies
    \[f(\nu)=|S\cap \nu|+r(|T\cap \nu|-1),\]
    and hence every $\nu\in A_F$ has
    \[\al(\nu)=\frac{(r-1)(|S\cap \nu|-1)}{r(e(\nu)-|T\cap \nu|)+1-|S\cap \nu|}.\]
\end{lem}


\section{Balanced Supersaturation}\label{sec:supersaturation}

Ever since the foundational work of Morris and Saxton \cite{morris2016number}, it has become standard  to prove upper bounds on random Tur\'an numbers by proving an ``edge--balanced supersaturation result.''  Informally, such a result says that in any graph $G$ with much more than $\ex(n,F)$ edges, one can find a collection $\c{H}$ of copies of $F$ such that no set of \textit{edges} of $G$ appears in too many copies of $\c{H}$.  Such a result combined with the powerful method of hypergraph containers quickly leads to effective bounds on $\ex(G_{n,p},F)$.

While we ultimately require an edge--balanced supersaturatoin result to prove our theorems, it will be convenient for us to initially prove a ``vertex--balanced supersaturation result'' which asks to find a collection $\c{H}$ of copies of $F$ such that no set of \textit{vertices} of $G$ appears in too many copies of $\c{H}$, where the exact notion of ``too many'' depends on the structure of the vertex sets under consideration.  Such a result immediately translates into an analogous edge--balanced supersaturation result, and by taking this vertex perspective at the start we will be able to prove our desired bounds inductively.  The idea of using vertex--balanced supersaturation first appeared implicitly in the work of Morris and Saxton~\cite{morris2016number} with it being made explicit by McKinley and Spiro~\cite{mckinley2023random} and later used in work of Nie and Spiro~\cite{nie202X}.

In order to have any hope of proving a balanced supersaturation result, we first need to know how to prove a supersaturation result.  For our setting, Conlon, Fox, and Sudakov \cite{conlon2010approximate} used dependent random choice to give a supersaturation result for $r$-semi-bounded graphs.  Very roughly, their idea is to build copies of $F$ by starting with some ``typical'' high-degree vertex $y^*\in V(G)$ to play the role of $v^*$, then choosing neighbors of $y^*$ to play the role of $S$, and then choosing the vertices to play the rest of $T$ from there.  While this is roughly the same line of argument we use here, we need to be extra cautious in our approach.  In particular, because we ultimately aim to construct a balanced collection of copies where no vertex in $G$ is in too many copies of $F$, we will need to impose some strong regularity conditions on the vertices of $G$ which were not needed in \cite{conlon2010approximate}.

\subsection{Vertex Supersaturation}


To state our main vertex balanced supersaturation statement, we define a function which measures how ``balanced'' a given set of vertices needs to be.

\begin{defn}
    Given an $r$-semi-bounded $F$ and a set of parameters $q,n,\del$, we define the function $D:2^{V(F)}\to \R\cup \{\infty\}$ by setting $D(\nu)=\infty$ if $F[\nu]$ has no edges, and otherwise we set
    \[D(\nu)=\del^{-|\nu|}q^{e(F)-f(\nu)}n^{v(F)-|\nu|}.\]
\end{defn}	

\begin{thm}\label{findingacollection}
If $F$ is an $r$-semi-bounded graph, then there exists sufficiently small $\ve, \delta>0$ and sufficiently large $C>0$ depending on $F$ such that the following holds. Let $G$ be a graph such that $q := \frac{e(G)}{n^2}$ satisfies $q \geq C n^{- 1 /r}$. Then, there exists a collection $\Phi$ of embeddings $\varphi: F \rightarrow G$ such that the following two conditions hold:
\begin{enumerate}
\item $|\Phi| \geq \ve q^{e(F)} n^{v(F)}$
\item For any $\nu \subseteq V(F)$ and any embedding $\psi: F[\nu] \rightarrow G$, the number of $\varphi \in \Phi$ such that $\varphi|_{F[\nu]} = \psi$ is less than $D(\nu)$. 
\end{enumerate} 
\end{thm}

To prove this we make the following auxiliary definitions.

\begin{defn}
    Let $G$ be a graph with $n$ vertices and $qn^2$ edges, and let $\Phi$ be a collection of embeddings of $F$ into $G$. Given some partial embedding $\psi:F[\nu]\to G$,   define \[\deg_{\Phi}(\psi) = |\{\varphi \in \Phi: \varphi|_{\nu} = \psi\}|.\] We say $\Phi$ is \emph{$D$-good} if for any $\nu \subseteq V(F)$ and partial embedding $\psi: F[\nu] \rightarrow G$ into $G$, we have $\deg_{\Phi} (\psi) \leq D(\nu)$. 
    
    We say that a partial embedding $\psi: F[\nu] \rightarrow G$ is \textit{saturated} (with respect to $\Phi$) if we have $\deg_{\Phi}(\psi) \geq \frac{1}{2}D(\nu)$. 
    We will additionally say that a subgraph $K \subseteq E(G)$ is a \textit{saturated $T$-star} if $K$ is a star and if there exists a star $K'\sub F$ with center in $T$ and an isomorphism $\psi:K' \rightarrow K$ with $\deg_{\Phi}(\psi)\ge \half D(V(K'))$.  If $K$ is a star with one edge we will refer to this simply as a \textit{saturated edge}.
\end{defn}  

In light of these definitions, we prove some results showing that $D$-good collections have few saturated structures.  We begin with a standard double counting result.

\begin{lem}\label{double counting}
    Let $F$ be an $r$-semi-bounded graph, $G$ a graph with $n$ vertices and $qn^2$ edges, and $\Phi$ a $D$-good collection.
    \begin{itemize}
        \item[(a)] For each $\nu\sub V(F)$ with $e(F[\nu])>0$, let $\xi_\nu$ denote the number of partial embeddings $\psi: F[\nu] \rightarrow G$ such that $\deg_{\Phi}(\psi) \geq \frac{1}{2} D(\nu)$.  Then
        \[\xi_\nu\le 2\frac{|\Phi|}{D(\nu)}.\]
        \item[(b)] For each $\nu',\nu \sub V(F)$ with $\nu'\sub \nu$ and $e(F[\nu])>0$, and for each partial embedding $\psi':F[\nu']\to G$; let $\xi_{\psi',\nu}$ denote the number of partial embeddings $\psi: F[\nu] \rightarrow G$ such that $\deg_{\Phi}(\psi) \geq \frac{1}{2} D(\nu)$ and such that $\psi|_{F[\nu']}=\psi'$.  Then
        \[\xi_{\psi',\nu}\le 2 \frac{D(\nu')}{D(\nu)}.\]
    \end{itemize}
\end{lem}
\begin{proof}
    For (a), let $\mu_{\nu}$ be the number of pairs $(\varphi, \psi)$ with $\psi: F[\nu] \rightarrow G$, $\varphi \in \Phi$,  $\deg_{\Phi}(\varphi) \geq \frac{1}{2} D(\nu)$, and $\varphi|_{F[\nu]} = \psi$.   
    Observe then that for any  $\nu \subseteq F$ we have
    \begin{equation}\frac{1}{2}D(\nu)\xi_{\nu} \leq  \mu_{\nu} \leq |\Phi|,\label{eq:double counting}\end{equation}
    with the upper bound using that specifying $\varphi\in \Phi$ for the pair specifies $\psi$, and the lower bound using that each $\psi$ counted by $\xi_\nu$ belongs to at least $\half D(\nu)$ pairs by definition.   Because $e(F[\nu])>0$ we have $D(\nu)\ne \infty$, and therefore we can divide both sides of \eqref{eq:double counting} by $\half D(\nu)$ to obtain the first result.

    For (b), let $\mu_{\psi',\nu}$ be the number of pairs $(\varphi, \psi)$ with $\psi: F[\nu] \rightarrow G$, $\varphi \in \Phi$,  $D_{\Phi}(\varphi) \geq \frac{1}{2} D(\nu)$, and $\varphi|_{F[\nu]} = \psi$ and $\psi|_{F[\nu]}=\psi'$.  In this case we have
    \[\frac{1}{2}D(\nu)\xi_{\nu,\psi'} \leq  \mu_{\nu,\psi'}\le D(\nu'),\]
    where now the upper bound uses that the number of $\varphi\in \Phi$ which agree with $\psi'$ on $\nu'$ is at most $\deg_{\Phi}(\psi')\le D(\nu')$ since $\Phi$ is $D$-good.  This gives the second result, finishing the proof.
\end{proof}

This in turn gives an analog of \Cref{lem:remove}.

\begin{cor}\label{standcount}
    Let $F$ be an $r$-semi-bounded graph, $G$ a graph with $n$ vertices and $qn^2$ edges, and $\Phi$ a $D$-good collection with $|\Phi| < \ve q^{e(F)} n^{v(F)}$. Then the following statements holds: 
    \begin{enumerate}
        \item[(a)] The number of saturated edges of $G$ is at most $2 e(F) \ve \delta^2 e(G)$.  
        
        \item[(b)] If $\nu \sub V(F)$ and $u\in S \cap \nu$ is such that $F[\nu\sm \{u\}]$ contains at least one edge, then for any partial embedding $\psi'$ of $\nu\sm \{u\}$, the number of embeddings $\psi$ of $F[\nu]$ with $\deg_{\Phi}(\psi)\ge \half D(\nu)$ and with $\psi|_{F[\nu\sm \{u\}]}=\psi'$ is at most $2\del qn$.

        \item[(c)] If $\nu \sub V(F)$ and if there exists a vertex in $T \cap \nu\sm \{v^*\}$ which is not incident to every edge of $F[\nu]$, then there exists some $v\in T\cap \nu\sm\{v^*\}$ such that for any partial embedding $\psi'$ of $\nu\sm \{v\}$, the number of embeddings $\psi$ of $F[\nu]$ with $\deg_{\Phi}(\psi)\ge \half D(\nu)$ and with $\psi'|_{F[\nu\sm \{v\}]}=\psi$ is at most $2\del q^{\deg_F(v)} n$.

    \end{enumerate} 
\end{cor}

\begin{proof}
    For (a), we observe by definition that for each saturated edge $e$ that there exists a partial embedding $\psi:F[\nu]\to G$ with $F[\nu]$ an edge and with $\psi$ a map counted by $\xi_\nu$.   By definition of $D$ and \Cref{lem:remove}(a), any $\nu$ with $F[\nu]$ an edge has $D(\nu)=\del^{-2} q^{e(F)-1}n^{v(F)-2}$.  As such, \Cref{double counting}(a) implies the number of saturated edges is at most
    \[\sum_{\nu:F[\nu]\cong K_2} \xi_\nu\le e(F)\cdot 2 |\Phi| \del^{2} q^{1-e(F)} n^{2-v(F)}< e(F)\cdot 2 \ve \del^2 e(G),\]
    with this last step using our hypothesis on $|\Phi|$.  This proves (a).

    For (b), the quantity we wish to upper bound is exactly
    \[\xi_{\psi',\nu}\le 2 \frac{D(\nu\sm \{u\})}{D(\nu)}=2 \del q^{f(\nu)-f(\nu\sm \{u\})} n=2\del q n,\]
    with the inequality used \Cref{double counting}(b), the first equality implicitly used that $\nu\sm \{u\}$ contains at least one edge (so that $D(\nu\sm \{u\})\ne \infty$), and the second equality used \Cref{lem:remove}(b).

    For (c), let $v$ be the vertex guaranteed by \Cref{lem:remove}(c).  Again, the quantity we wish to upper bound is exactly
    \[\xi_{\psi',\nu}\le 2 \frac{D(\nu\sm \{u\})}{D(\nu)}=2 \del q^{f(\nu)-f(\nu\sm \{u\})} n=2\del q^{\deg_F(v)} n,\]
    proving the result.
\end{proof}

We now prove our main technical lemma which allows us to inductively build our $D$-good collection $\Phi$.

\begin{lem}\label{inductivestep}
If $F$ is an $r$-semi-bounded graph, then there exist $\ep,\del,C>0$ such that the following holds.  If $G$ is a graph with $n$ vertices and $qn^2 \geq Cn^{2 -1 /r}$ edges, and if $\Phi$ is a $D$-good collection with $|\Phi| < \ve  q^{e(F)}n^{v(F)}$, then there exists an embedding $\varphi$ of $F$ into $G$ such that $\varphi \not \in \Phi$ and $\Phi \cup \{\varphi\}$ is $D$-good.
\end{lem}
\begin{proof}
Our strategy  will be to find a large set $\c{A}$ of copies of $F$ which avoids saturated $T$-stars and edges. We then show that in this set there are few elements $\c{B}$ which contain any saturated set at all. Then, since $\c{A} \setminus \c{B}$ is large, there is an element of $\c{A} \setminus \c{B}$ which is not in $\Phi$, and such an element will satisfy the conditions of the lemma.  In what follows, we let 
\begin{align*}
    \eta &= \frac{1}{2^{s + 1}r^2},  & &\text{ } & \ve &= \frac{1}{2^{v(F) + 4 + 3e(F) }}\eta^{2|T|},  \\
    \delta &= \min\left\{\frac{\varepsilon}{2^{v(F) + 3r + 3}}, \frac{\eta}{2^{v(F) + 8}r^{r + 2}}\right\},  & &\text{ and } & C &=8 \max\left\{\left( \frac{8v(F)}{ \eta^2} \right)^{1 / r}, \delta^{- v(F)} 2^{v(F)} \right\}. 
\end{align*} 

We begin with a simple cleaning argument, where here and throughout we omit floors and ceilings for ease of presentation. 
\begin{claim}
    There exists a bipartite subgraph $H \subseteq G$ with bipartition $(X, Y)$ such that none of the edges of $H$ are saturated, $e(H)\ge  \frac{1}{16}q v(H)^2$, $|Y|\ge \half v(H)$, and every vertex in $Y$ has degree exactly $\frac{1}{8}qn$.  
\end{claim}
\begin{proof}
        Let $G'$ be $G$ after deleting all of its saturated edges.  By \Cref{standcount}(a), there are no more than $4 e(F) \delta^2 \ve qn^2$ saturated edges, so by choice of $\delta, \ve$ we have $e(G') \geq \frac{1}{2}qn^2$.  Let $G''\sub G'$ be a bipartite subgraph with at least half of the edges of $G'$, which means $e(G'')\ge \frac{1}{4} qn^2$.  Now let $G'''$ be the graph obtained from $G''$ by iteratively removing vertices of degree less than $\frac{1}{8}qn$. We remove no more than $\frac{1}{8}qn^2$ edges in this process, so $G'''$ is a bipartite graph with at least $\frac{1}{8}qn^2$ edges and minimum degree at least $\frac{1}{8} qn$.  
        
        Let $(X,Y)$ be a bipartition of $G'''$, and without loss of generality we may assume $|X|\le |Y|$.  Delete edges from $G'''$ so that every vertex of $Y$ has degree exactly $\frac{1}{8}qn$ and let $H$ be the resulting graph.  Since $H\sub G'$ this graph has no saturated edges, and by construction it has $e(H)=\frac{1}{8} qn |Y|\ge \frac{1}{16} q v(H)^2$, giving the result.
\end{proof}
From now on we let $m:=v(H)$ and $\rho:= \frac{1}{8} q n m^{-1}$ where we think of $m,q$ as playing the analogous roles of $n,q$ for $H$.  Observe that $\rho\ge \frac{1}{8}q$ and for all $y \in Y, \deg_H(y) = \rho m$. 
 
    Given an $a$-tuple of vertices $\mathbf{x}$, we let $d(\mathbf{x})$ be the number of vertices in the common neighborhood of $\mathbf{x}$ and we let  $d^*(\mathbf{x})$ denote the number of vertices $y \in N(\mathbf{x})$ such that the star formed from the vertices of $\mathbf{x}$ and $y$ contains some substar which is a saturated $T$-star. 
    
    \begin{claim}\label{dangeroussmall}
        $\sum_{\mathbf{x} \in X^a} d^*(\mathbf{x}) \leq 4^{v(F)+2}a^{a+1} \del \rho^a m^{a+1}$. 
    \end{claim}
    Note that the related sum $\sum_{\mathbf{x} \in X^a} d(\mathbf{x})$ is roughly equal to $\rho^a m^{a+1}$ since for each of the at most $m$ vertices $y\in Y$ there are exactly $\rho^a m^a$ tuples $\mathbf{x}$ with $y\in N(\mathbf{x})$ since $\deg_H(y)= \rho m$.  The point of this claim then is that by restricting to $d^*(\mathbf{x})$ we reduce this total degree sum by a factor of roughly $\del$.
    \begin{proof}
        We first show that the number of saturated $\psi:F[\nu]\to V(H)$ such that $F[\nu]$ is a $K_{1,t}$ with center in $T$ and such that the center of $F[\nu]$ is mapped into $Y$ is at most  $4^{v(F)+2} \del \rho^t m^{t+1}$ for all $1\le t\le a$.  Indeed, the result is trivial for $t=1$ since $H$ contains no saturated edges by construction, so we may assume $2\le t\le a$ from now on.  Choose $\nu$ in at most $2^{v(F)}$ ways so that $F[\nu]$ is a $K_{1,t}$ with center in $T$, and let $u\in \nu$ be an arbitrary leaf which can also be chosen in at most $2^{v(F)}$ ways.  We can trivially map the center of $\nu$ in at most $m$ ways for it to be mapped into $Y$, and from there we can map the $t-1$ vertices of $\nu\sm \{u\}$ in at most $\rho^{t-1}m^{t-1}$ ways since each vertex of $Y$ has degree exactly $\rho m$ in $H$.  Let $\psi'$ denote this embedding of $\nu\sm \{u\}$.  By \Cref{standcount}(b), the number of saturated embedddings $\psi$ of $\nu$ which extend $\psi'$ is at most $2\del qn=16 \del \rho m$.  In total then we conclude that the number of choices for $\psi$ is at most
        \[4^{v(F)}\cdot m\cdot \rho^{t-1} m^{t-1}\cdot 16 \del \rho m= 4^{v(F)+2} \del \rho^t m^{t+1},\]
        proving the subclaim.
        
        Returning to the main claim, we ultimately need to bound the number of pairs $(\mathbf{x},y)$ such that $y\in Y\cap N(\mathbf{x})$ and such that the star formed by $y$ and the vertices of $\mathbf{x}$ contains a saturated $T$-substar $K_{1,t}$.  The total number of saturated $T$-substars $K_{1,t}$ is at most  $4^{v(F)+2} \del \rho^t m^{t+1}$ since by definition any such substar must have a saturated mapping to it of the form discussed in the subclaim above.  Given such a substar, the choice for $y$ as well as $t$ of the vertices in $\mathbf{x}$ is determined, and the number of choices for the remaining $a-t$ vertices of $\mathbf{x}$ is at most $\rho^{a-t}m^{a-t}$ since $y$ has degree exactly $\rho m$.  There are then at most $a!\le a^a$ choices for the sequence $\mathbf{x}$ containing all of these vertices, so in total the number of choices for $(\mathbf{x},y)$ with a given value of $t$ is at most $4^{v(F)+2}a^a \del \rho^a m^{a+1}$ and summing this over all $a$ values of $t$ gives the desired bound.
    \end{proof}

We say an $a$-tuple $\mathbf{x}$ is \textit{light} if $d(\mathbf{x}) \leq \eta^2 \rho^a m$. We say an $a$-tuple $\mathbf{x}$ is \textit{dangerous} if  $d^*(\mathbf{x}) \geq  \frac{1}{2} d(\mathbf{x})$. 
\begin{claim}
    If $\mathbf{x}$ is an $a$-tuple which is neither light nor dangerous, then there exists a set $N'(\mathbf{x})\sub N(\mathbf{x})$ with
    \[\half \eta^2 \rho^a m\le |N'(\mathbf{x})|\le \rho^a m\]
    such that for all $y\in N'(\mathbf{x})$, the star formed with $y$ and the vertices of $\mathbf{x}$ contains no saturated $T$-star as a subgraph.
\end{claim}
\begin{proof}
    By definition, the set $N''(\mathbf{x})\sub N(\mathbf{x})$ of vertices $y$ such that the star formed with $y$ and the vertices of $\mathbf{x}$ contains no saturated $T$-star as a subgraph is a set of size
    \[d(\mathbf{x})-d^*(\mathbf{x})\ge \half d(\mathbf{x})\ge \half \eta^2 \rho^a m.\]
    Taking any subset of $N''(\mathbf{x})$ of size at most $\rho^a m$ then gives the desired result.
\end{proof}

From now on we fix some specific choice of $N'(\mathbf{x})$ for each such $\mathbf{x}$ as in the claim.  Ultimately, when building our large collection of copies of $F$, we will want to embed $v^*$ into a vertex $y^*\in Y$ such that most of the tuples in $N(y^*)$ are neither light nor dangerous.  To this end, for an integer $a$ we say a vertex $y$ is light-unembeddable at $a$ if the number of light $a$ tuples in  $N(y)$ is greater than $\eta \rho^am^a$. We say a vertex $v$ is dangerous-unembeddable at $a$ if the number of dangerous $a$-tuples in $N(v)$ is greater than $\eta \rho^am^a$. We say $y\in Y$ is embeddable if there exists no $1\le a\le r$ such that $y$ is either light-unembeddable or dangerous-unembeddable at $a$. 

\begin{claim}\label{manyembeddable}
    There are at least $\frac{1}{4} m $ embeddable vertices. 
\end{claim}
\begin{proof}

Let $\gamma_{\ell, a}$ be the number of $y$ which are light-unembeddable at $a$.  We count $\mu_{\ell, a}$, the number of pairs $(\mathbf{x},y)$ where $y$ is light-unembeddable at $a$ and $\mathbf{x}$ is a tuple in the neighborhood of $y$ which is light, i.e.\ which has $d(\mathbf{x}) \leq \eta^2 \rho^a m $.  Then 
\begin{align*}
     \mu_{\ell, a} &\leq \sum_{\mathbf{x} \text{ light}} d(\mathbf{x})\\
     &\leq \eta^2 \rho^am^{a + 1}  
\end{align*}

On the other hand, 
 \begin{align*}
     \mu_{\ell, a} &\geq \sum_{y \text{ light-unembeddable}} \eta \rho^a m^a\\ &= \eta \rho^a m^a \gamma_{\ell, a}. 
 \end{align*}

Therefore, $\gamma_{\ell, a} \leq \eta m$.  Similarly, let $\gamma_{d, a}$ be the number of vertices which are  dangerous-unembeddable at $a$ and $\mu_{d, a}$ the number of pairs $(\mathbf{x},y)$ with $y$ dangerous-unembeddable at $a$ and $\mathbf{x}$ an $a$-tuple in the neighborhood of $y$ which is dangerous, i.e.\ which is such that $d^*(\mathbf{x}) \geq \frac{1}{2} d(\mathbf{x})$. Then, by Claim~\ref{dangeroussmall},
\begin{align*}
   \mu_{d, a} &\leq \sum_{\mathbf{x} \text{ dangerous}} d(\mathbf{x}) \\
   &\leq 2 \sum_{\mathbf{x} \text{ dangerous}} d^*(\mathbf{x}) \\
   &\leq 2 \sum_{\mathbf{x}  \in X^a} d^*(\mathbf{x})\\
   &\leq 2^{2v(F)+5}a^{a+1} \del \rho^a m^{a+1}.
\end{align*}

On the other hand, the same reasoning as above gives 
$$\mu_{d, a} \geq \eta \rho^a m^a \gamma_{d, a}.$$

Thus, $\gamma_{d, a} \leq \eta^{-1} 2^{2v(F)+5}a^{a+1} \del m$. Summing over all $1 \leq  a \leq r$, we have that the number of embeddable vertices is at least \[|Y|- \eta r m - \eta^{-1} 2^{2v(F)+5}r^{r+ 2} \del m \geq \frac{1}{4} m,\] with this last step using $|Y|\ge m/2$ by construction and our choice of $\delta, \eta$, giving the desired result.

\end{proof}

We also observe the following.

\begin{claim}
    If $y\in Y$ is embeddable, then there exist at least $2^{-|S| - 1} \rho^{|S|} m^{|S|}$ tuples $(x_1,\ldots,x_{|S|})\in N(y)^{|S|}$ of distinct vertices such that for all $1\le a\le r$ no $a$-tuple of these vertices is either light or dangerous.
\end{claim}
\begin{proof}
    Form an auxiliary hypergraph $\c{N}$ on $N(y)$ where a set of size $1\le a\le r$ is an edge if it is either light or dangerous.  Observe then that the tuples we wish to find are exactly ordered independent sets of size $|S|$ in $\c{N}$.  For each hyperedge of $\c{N}$ of size $1\le a\le r$, we trivially have that the number of tuples $(x_1,\ldots,x_{|S|})$ which contain all the vertices of this hyperedge is at most $|S|^a v(\c{N})^{|S|-a}=|S|^a(\rho m)^{|S|-a}$.  
      By definition of $y$ being embeddable, the number of hyperedges of $\c{N}$ of size $a$ is at most $2\eta \rho^a m^a$, so in total the number of tuples $(x_1,\ldots,x_{|S|})$ of vertices which do not form an independent set is at most
    \[\sum_{a=1}^r 2 |S|^r \eta \rho^{|S|} m^{|S|}= 2r |S|^r \eta \rho^{|S|} m^{|S|}.\]
    On the other hand, the number of tuples of distinct vertices is at least
    \[(v(\c{N})-|S|)^{|S|}=(\rho m-|S|)^{|S|}\ge 2^{-|S|} \rho^{|S|} m^{|S|},\]
    with this last inequality using $\rho m \geq \frac{1}{8} qn\ge 2 |S|$ since $C$ is sufficiently large.  Therefore, by choice of $\eta$, there are at least $2^{-|S| - 1}\rho^{|S|}m^{|S|}$ many tuples of distinct vertices containing no $a$-tuple of vertices which is either light or dangerous. 
\end{proof}

We now construct a large collection $\c{A}$ of embeddings $\varphi$ of $F$ as follows.  We start by setting $\varphi(v^*)$ to be any of the at least $\quart m$ embeddable vertices $y^*\in Y$.  We then pick some tuple $(x_1,\ldots,x_{|S|})\in N(y^*)^{|S|}$ of distinct vertices such that no $a$-tuple of these vertices with $1\le a\le r$ is either light or dangerous and then embed each vertex of $S$ into one of these $x_i$ vertices in an arbitrary order.  Finally, for each $v\in T\sm \{v^*\}$, say with $S_v\sub S$ its neighborhood in $F$, we choose $\varphi(v)$ to be any vertex in $N'(\varphi(S'))$ which is not equal to any other vertex in the image.  It is not difficult to see that each $\varphi$ constructed in this way is an embedding of $F$ and that the total set $\c{A}$ of embeddings constructed in this way satisfies
\begin{equation}|\c{A}|\ge \quart m\cdot  2^{-|S| - 1} \rho^{|S|} m^{|S|}\cdot \prod_{v\in T\sm \{v^*\}} (\half \eta^2 \rho^{\deg_F(v)} m-v(F))\ge \frac{\eta^{2|T|}}{2^{|T| + |S| + 3}} \rho^{e(F)} m^{v(F)},\label{eq:size of A}\end{equation}
with this last inequality using $\frac{\eta^2}{4} \rho^{\deg(v_i)} m \geq \frac{\eta^2}{4}(\frac{1}{8}q)^{\deg(v_i) - 1} (\frac{1}{8} qn) \geq 2v(F)$ by choice of $C$, and also that $|S|+\sum_{v\in T\sm \{v^*\}} \deg_F(v)=e(F)$ since $\deg_F(v^*)=|S|$.
    

We seek to show that there exists some $\varphi\in \c{A}$ such that we can add $\varphi$ to $\Phi$. We will ultimately be able to do this if $\varphi\notin \Phi$ and if there is no $\nu \subseteq V(F)$ such that $D_{\Phi}(\varphi|_{F[\nu]}) \geq \frac{1}{2}D(\nu)$. In this vein, we define \[\c{B}_\nu = \{ \varphi \in \A: \deg_{\Phi}(\varphi|_{F[\nu]}) \geq \frac{1}{2}D(\nu) \},\] and $\c{B} = \bigcup_{\nu \subseteq V(F)} \c{B}_\nu$. 
\begin{claim}
    $|\c{B}|\le 2^{v(F) + 3r+  1}  \del \rho^{e(F)}m^{v(F)}$.
\end{claim}

\begin{proof}
    We will show $|\c{B}_\nu|\le 2^{3r+1} \del \rho^{e(F)} m^{v(F)}$ for all $\nu\sub V(F)$, from which the result will follow.  We prove this through some case analysis on $\nu$.  Implicitly whenever we are in Case $i$ we will assume that we are not in any previous Case $j$ with $j<i$.

    \textbf{Case 1:} $\nu$ is an edge of $F$. In this case $\c{B}_\nu=\emptyset$ since $H$ has no saturated edges.

    \textbf{Case 2:} 
    $\nu$ contains a vertex in $T \cap \nu\sm \{v^*\}$ which is not incident to every edge of $F[\nu]$.  Let $v\in T\cap \nu\sm\{v^*\}$ be a vertex guaranteed by \Cref{standcount}(c)
    
    
    We will upper bound $|\c{B}_\nu|$ by counting all the ways we have to pick the vertex playing the role of $v^*$, 
     then the vertices playing the role of $S$, then $T\sm \{v, v^*\}$, then finally $v$.  Trivially, the number of choices for $y^*$ playing the role of $v^*$ is at most $m$, and by our bounded degree condition on $H$, the number of choices for the vertices of $S$ is at most $(\rho m)^{|S|}$.  Once these are picked, each $w\in T\sm \{v^*,v\}$ has its neighborhood $S_v=N_F(w)$ embedded, so $\varphi(w)$ must lie in $N'(\varphi(S_v))$, which by definition of $N'$ means the number of choices for $w$ is at most $\rho^{\deg_F(w)}m$.  It only remains now to bound the number of choices for $v$.

    Observe that at this point we have already specified $\varphi$ on all the vertices of $\nu\sm \{v\}$ since $\nu\sm \{v\}\sub S\cup T\sm \{v\}$.  Thus by \Cref{standcount}(c) and the fact that $\Phi$ is good with respect to $D$, the number of possible $y$ we could choose to play the role of $v$ so that $\varphi|_{F[\nu]}$  is saturated (i.e.\ so that $\varphi\in \c{B}_\nu$) given that we have already specified $\varphi|_{F[\nu\sm \{v\}]}$ is at most $2 \del q^{\deg_F(v)} n \leq 2^{3r + 1} \delta \rho^{\deg_F(v)} m$.  In total then we find that
    \[|\c{B}_\nu|\le m\cdot \rho^{|S|}m^{|S|}\cdot 2^{3r + 1} \delta \prod_{w\in T\sm \{v^*\}} \rho^{\deg_F(w)}m= 2^{3r+1} \delta \rho^{e(F)} m^{v(F)}, \] giving the desired bound. 

    \textbf{Case 3:} $\nu$ contains a vertex $u\in S$ with $N_F(u)\cap \nu\sub \{v^*\}$ (that is, $u$ is either isolated in $\nu$ or adjacent only to the dominating vertex $v^*$).  
    Similar to the previous case, we will upper bound $|\c{B}_\nu|$ by counting all the ways we have to pick the vertex playing the role of $v^*$,then the vertices playing the role of $S\sm \{u\}$, then $T\sm N_F(u)$, then $u$, then $N_F(u)\sm \{v^*\}$. As before we have $m$ choices for $v^*$, $(\rho m)^{|S| - 1}$ choices for the vertices of $S\sm \{u\}$, and for each $ v \in T \sm N_F(u)$, we have no more than $\rho^{\deg_F(v)}m$ choices.  Crucially at this point we have specified $\varphi$ on every vertex of $\nu\sm \{u\}$ by assumption of the case we are in.  Thus by \Cref{standcount}(b), the number of choices for $u$ such that $\varphi|_{F[\nu]}$ is saturated is at most $2 \delta qn  < 16 \delta \rho m$. Given $u$, there are no more than $\rho^{\deg_F(v)}m$ choices for each vertex $v \in N_F(u) \setminus v^*$, which combined with the previous counts finishes this case.
    
    \textbf{Case 4:} We are not in any of the cases above. Observe that this case implies that $\nu$ consists of some $w\in T\sm \{v^*\}$ together with some subset of its neighbors.  Indeed, not being in Case 2 implies that every edge in $\nu$ is incident to a single vertex $w\in T$, and not being in Case 3 implies both that $w\ne v^*$ (since otherwise every $v\in S\cap \nu$ would have $N_F(v)\cap \nu\sub \{w\}=\{v^*\}$) and that every vertex in $S\cap \nu$ must be adjacent to $w$, giving the stated observation. In particular, every $\varphi \in B_{\nu}$ has the image of $\varphi|_{F[\nu]}$ being a saturated $T$-star of $H$.  But by construction, every $\varphi \in \c{A}$ is such that $\varphi(w)\in N'(\varphi(N_F(w)))$ and by definition this implies that $\varphi|_{F[\nu]}$ can not be a saturated $T$-star.  This gives a contradiction, showing that in fact $\c{B}_\nu=\emptyset$ in this case. Summing up over all $\nu$, the claim follows. 
\end{proof}

We now seek to find a $\varphi$ in $\A$ which we can add to $\Hc$ and still have a $D$-good collection. Recalling our bound on $|\c{A}|$ from \eqref{eq:size of A}, we find $|\A| \geq 2|\c{B}|$ by the claim and our choice of $\delta$, and in particular, $|\A \setminus \c{B}| > 2^{3 e(F)} \ve \rho^{e(F)}m^{v(F)}$ by our choice of $\ve$. Therefore, $|\A \setminus \c{B}| > \ve q^{e(F)}n^{v(F)} \geq |\Hc|$, so there is some $\varphi \in \A \setminus (\c{B} \cup \Hc)$. By definition of $\c{B}$, we have that $\Phi\cup \{\varphi\}$ is $D$-good, proving the result.  
    \end{proof}
With this we can prove our main vertex balanced supersaturation result.  
\begin{proof}[Proof of \Cref{findingacollection}]
Suppose the conditions of \Cref{findingacollection} holds. Take any maximal family $\cH$ of embeddings of $F$ satisfying condition $2$. By maximality, there is no $\varphi: F \rightarrow G$ such that $\varphi \not \in \cH$ and $\cH \cup \{\varphi\}$ is $D$-good. By \Cref{inductivestep} we must have $|\cH| \geq \ve q^{e(F)}n^{v(F)}$, proving the result.
\end{proof}


\subsection{Edge Supersaturation}

Ultimately we need an \textit{edge} balanced supersaturation result to work with hypergraph containers.  We will use a simple translation of our vertex balanced supersaturation result to give such a statement.
\begin{thm}\label{thm:balancedSupersatEdges}
    If $F$ is $r$-semi-balanced, then there exists some $C>0$ such that if $G$ is an $n$-vertex graph with $qn^2$ edges such that $qn^2\ge C n^{2-1/r}$, then there exists a hypergraph $\c{H}$ on $E(G)$ whose hyperedges are copies of $F$ in $G$ and is such that $|\c{H}| \ge C^{-1} q^{e(F)} n^{v(F)}$ and such that for every $\sig\sub E(G)$ with $1\le |\sig|\le e(F)$, we have
    \[\deg_{\c{H}}(\sig)\le C q^{e(F)-1}n^{v(F)-2}\left(\max_{\nu\sub V(F): \del(F[\nu])\ge 1,\ e(\nu)\ge 2}[q^{1-f(\nu)}n^{2-|\nu|}]^{1/(e(\nu)-1)}\right)^{|\sig|-1}.\]
\end{thm}
To give some partial justification to this strange looking function, we note that because $f(\nu)\ge e(\nu)$ and $q\le 1$, this maximum is always at least $q^{-1} \max n^{(2-|\nu|)/(e(\nu)-1)}=q^{-1} n^{-1/m_2(F)}$, so the bound we get is always no better than $q^{-1} n^{-1/m_2(F)}$.  


\begin{proof}
    Let $F$ be an $r$-semi-bounded graph. By \Cref{findingacollection}, there exists a collection $\Phi$ of embedddings $\varphi:F \rightarrow G$ such that the following two conditions hold: 
    \begin{enumerate}
        \item $|\Phi| \geq \ve q^{e(F)}n^{v(F)}.$
        \item For any $\nu \subseteq V(F)$ and any $\psi: F[\nu] \rightarrow G$, the number of $\varphi \in \Phi$ such that $\varphi|_{F[\nu]} = \psi$ is less than $D(\nu)$. 
    \end{enumerate}

    Given a set of edges $\sig\sub E(G)$, we say that a partial embedding $\psi: F[\nu]\to G$ is a $\sig$-embedding if the image of $\psi$ equals $V(\sig)$ (which is the set of vertices used in an edge of $\sig$) and if for each $e\in \sig$ there exists an edge of $F[\nu]$ which maps onto $e$.  Let $\c{H}$ be the hypergraph formed on $E(G)$ where a set $\sigma$ of $e(F)$ edges forms a hyperedge if there is some map $\varphi \in \c{G}$ which is a $\sig$-embedding.  For each set of edges $\sigma\subseteq E(G)$, let $\Psi_\sigma$ denote the set of $\sig$-embeddings.  
    
    As before we let $\deg_{\Phi}(\psi)$ denote the number of $\phi\in \Phi$ which restrict to $\psi$.   Is not difficult to see then that for any $\sigma \subseteq E(G)$ we have  $\deg_{\c{H}}(\sigma) \leq \sum_{\psi \in \Psi_\sig} \deg_{\Phi}(\psi)$.  In particular, if $|\sigma|=1$, say with $\sig=\{e\}$, then there are only $2e(F)$ possible $\sig$-embeddings and every $\sig$-embedding $\psi:F[\nu]\to G$ has $\nu$ an edge of $F$, hence \[\deg_{\Phi}(\psi)\le D(\nu)=\del^{-2} q^{e(F)-f(\nu)} n^{v(F)-2}=\del^{-2} q^{e(F)-1} n^{v(F)-2},\]
    with this last step using \Cref{lem:remove}(a).  In total this implies that for $|\sig|=1$ we have
    \[\deg_{\c{H}}(\sig)\le 2e(F) \del^{-2}  q^{e(F)-1} n^{v(F)-2},\]
    which gives the desired bound for $C$ sufficiently large.

    For $\sigma > 1$, every $\psi : F[\nu] \rightarrow G$ which is a $\sig$-embedding has $e(\nu)\ge |\sig|\ge 2$, and hence

    \begin{align*}
        \deg_{\Phi}(\psi) &\leq \delta^{ - |\nu|}q^{e(F) - f(\nu)}n^{v(F) - |\nu|}\\
        &\leq \delta^{- |\nu|}q^{e(F) - 1}n^{v(F) - 2} \left( q^{1  - f(\nu) } n^{2 - |\nu|}\right)\\
        &\leq  \delta^{- |\nu|}q^{e(F) - 1}n^{v(F) - 2} \left( q^{1  - f(\nu) } n^{2 - |\nu|}\right)^{\frac{e(\nu) - 1}{e(\nu) - 1}}\\
        &\leq  \delta^{- |\nu|}q^{e(F) - 1}n^{v(F) - 2} \left( q^{1  - f(\nu) } n^{2 - |\nu|}\right)^{\frac{|\sigma| - 1}{e(\nu) - 1}}\\
        &\leq \delta^{- v(F)} q^{e(F) - 1}n^{v(F) - 2}\left( \max_{\nu\sub V(F): \del(F[\nu])\ge 1,\ e(\nu)\ge 2}[q^{1-f(\nu)}n^{2-|\nu|}]^{1/(e(\nu)-1)}\right)^{|\sigma| - 1},
    \end{align*}
    with this second to last step using that $|\sig|\le e(\nu)$ and that the term inside the parenthesis is at most 1 since $q\ge n^{-1/r}$ and $f(\nu)\le r(|\nu|-2)+1$ by \Cref{lem:fBounds}; and the last step using that since $\psi$ is a $\sig$-embedding, $F[\nu]$ must have minimum degree at least 1. Summing this over the at most $2^{v(F)}v(F)!$ maps inside $\Psi_\sig$ gives the final result. 
\end{proof}

\section{Random Tur\'an Results}\label{sec:randomTuranResults}
\subsection{Using Balanced Supersaturation}

In this section we will use our balanced supersaturation result \Cref{thm:balancedSupersatEdges} together with standard machinery from hypergraph containers to give our main technical upper bounds on $\ex(G_{n,p},F)$. For this we need the following, where here for a hypergraph $\c{H}$ we let $\Del_i(\c{H})$ denote the maximum $i$-degree of $\c{H}$, i.e.\ the maximum number of edges containing a given set of $i$ vertices of $\c{H}$.

\begin{defn}
    Given positive functions $M=M(n),\ \gamma=\gamma(n)$, and $\tau=\tau(n,m)$, we say that a graph $F$ is \emph{$(M,\gamma,\tau)$-balanced} if for every $n$-vertex $m$-edge graph $G$ with $n$ sufficiently large and $m\ge M(n)$, there exists a non-empty collection $\c{H}$ of copies of $F$ in $G$ (which we view as an $e(F)$-uniform hypergraph on $E(H)$) such that, for all integers $1\le i\le e(F)$,
    \begin{equation}\label{equation:Delta}
    \Delta_i(\c{H})\le \frac{\gamma(n)|\c{H}|}{m}\l(\frac{\tau(n,m)}{m}\r)^{i-1}.     
    \end{equation}
\end{defn}

In particular, a simple translation of \Cref{thm:balancedSupersatEdges} yields the following.

\begin{cor}\label{cor:tauFormulation}
    If $F$ is $r$-semi-balanced, then there exists some $C>0$ such that $F$ is $(Cn^{2-1/r},C,\tau)$-balanced where $\tau(n,qn^2):=1$ for $q>1$, and otherwise
    \[\tau(n,qn^2):=qn^2 \max_{\nu\sub V(F): \del(F[\nu])\ge 1,\ e(\nu)\ge 2}[q^{1-f(\nu)}n^{2-|\nu|}]^{1/(e(\nu)-1)}.\]
\end{cor}
We note that the $q>1$ case is vacuously true since no graph $G$ has more than $n^2$ edges and this condition is needed only for minor technical reasons.  The motivation for this definition of balacedness comes from the following general result, the proof of which follows from a standard application of hypergraph containers.

\begin{prop}[\cite{nie2024random} Proposition 2.6 ]\label{Lemma:General Random Turan}
Let $F$ be a graph.  If there exists a  $C>0$ and positive functions $M=M(n)$ and $\tau=\tau(n,m)$ such that 
\begin{itemize}
    \item[(a)] $F$ is $(M,(\log n)^{C}, \tau)$-balanced, and
    \item[(b)] For all sufficiently large $n$ and $m\ge M(n)$, the function $\tau(n,m)$ is non-increasing with respect to $m$ and satisfies $\tau(n,m)\ge 1$, 
\end{itemize}
then there exists $C'\ge 0$ such that for all sufficiently large $n$, $m\ge M(n)$, and $0<p\le 1$ with $pm\rightarrow\infty$ as $n\rightarrow \infty$, we have a.a.s.
$$
\ex(G_{n,p},F)\le \max\l\{C'pm,\tau(n,m)(\log n)^{C'}\r\}.
$$
\end{prop}
This result is \textit{almost} enough to prove our results, except that it gives bounds which are off\footnote{For experts in containers: this is essentially because the argument of \Cref{Lemma:General Random Turan} does not utilize fingerprints.  In turn, we prove \Cref{Lemma:General Random Turan for Large p} using fingerprints.} by some logarithmic factors for $p \ge (\log n)^{-O(1)}$.  To get around this we need one more result.

\begin{prop}\label{Lemma:General Random Turan for Large p}
Let $F$ be a graph.  If there exists $\alpha, \beta, C>0$ and positive functions $q_0 = q_0(n)$ and $\tau=\tau(n,qn^2)$ such that 
\begin{itemize}
    \item[(a)] $F$ is $(q_0(n)n^2,C , \tau)$-balanced, and
    \item[(b)] For all sufficiently large $n$ and $q\ge q_0(n)$, we have $\tau(n,qn^2) = \max \{n^{2 - \alpha}, q^{1 - \beta}n \},$
\end{itemize}
then there exists $C'\ge 0$ such that for all sufficiently large $n$, $q\ge q_0(n)$, and $0<p\le 1$ with $pqn^2\rightarrow\infty$ as $n\rightarrow \infty$, we have a.a.s.
$$
\ex(G_{n,p},F)\le \max\l\{C'pqn^2,\ n^{2- \alpha}(\log n)^{C'},\ q^{1- \beta}n,\ p^{1 - \frac{1}{\beta}}n^{2 - \frac{1}{\beta}}\r\}.
$$
\end{prop}

The proof of \Cref{Lemma:General Random Turan} is a straightforward generalization of arguments used by Morris and Saxton \cite{morris2016number}, and as such we defer its proof to an arXiv only appendix.  To apply these results, we  need some facts about the $\tau$ function given by our supersaturation result.

\begin{lem}\label{lem:tauProperties}
    Let $F$ be an $r$-semi-bounded graph which contains a cycle, and for real numbers $q,n$ define $\tau(n,qn^2)=1$ for $q>1$ and otherwise
    \[\tau(n,qn^2)=qn^2 \max_{\nu\sub V(F): \del(F[\nu])\ge 1,\ e(\nu)\ge 2}[q^{1-f(\nu)}n^{2-|\nu|}]^{1/(e(\nu)-1)}.\]
    Then the following hold:
    \begin{itemize}
        \item[(a)] $\tau(n,qn^2)$ is non-increasing with respect to $q$ and has $\tau(n,qn^2)\ge 1$ for all $q$.
        \item[(b)] We have $\tau(n,qn^2)\le q^{1-r}n$ for all $n^{-1/r}\le q\le n^{\al(F)/r-1/r}$.
        \item[(c)] We have $\tau(n,qn^2)\le n^{2-1/m_2(F)}$ for all $q\ge n^{-\be(F)}$.
    \end{itemize}
\end{lem}
Note for (b) that $n^{\al(F)/r-1/r}\le 1$ due to \Cref{lem:alBound}.

\begin{proof}
    For (a), to show $\tau$ is non-increasing for $q\le 1$ it suffices to show that $qn^2[q^{1-f(\nu)}n^{2-|\nu|}]^{1/(e(\nu)-1)}$ is non-increasing in $q$ for all choices of $\nu$.  To prove that this is non-increasing, we observe that since $f(\nu)\ge e(\nu)$ by \Cref{lem:fvse}, the exponent for $q$ in this expression is at most 0, and hence the expression is non-increasing in $q$.  To show $\tau(n,qn^2)$ is non-increasing for $q\ge 1$, it suffices to show $\tau(n,n^2)\ge 1$, i.e.\ that at least one term in the maximum defining $\tau$ is at least 1 at $q=1$.    And indeed, taking $\nu=S\cup \{v^*\}$ shows
    \[\tau(n,n^2)\ge n^{2-\frac{2-|S\cup \{v^*\}|}{e(S\cup \{v^*\})-1}}=n \ge 1,\]
    where implicitly the first inequality uses that the set $\nu=S\cup \{v^*\}$ satisfies $e(\nu)\ge 2$, which must hold if $F$ contains a cycle.  This implies that $\tau$ is non-increasing in $qn^2$, and since $\tau(n,qn^2)=1$ for $q\ge 1$ this implies that $\tau(n,qn^2)\ge 1$ for all $q$.

    For (b), we note that having $\tau(n,qn^2)\le q^{1-r}n$ for $q\le 1$ is equivalent to saying that for all $\nu\sub V(F)$ with $\del(F[\nu])\ge 1$ and $e(\nu)\ge 2$ we have
    \[qn^2[q^{1-f(\nu)}n^{2-|\nu|}]^{1/(e(\nu)-1)}\le q^{1-r}n,\]
    which after raising both sides to the power of $e(\nu)-1$ and rearranging can be seen to be equivalent to proving
    \begin{equation}q^{r(e(\nu)-1)+1-f(\nu)}\le n^{|\nu|-2-(e(\nu)-1)}.\label{eq:alphaTau}\end{equation}
    First consider the subcase that $r(e(\nu)-1)+1-f(\nu)\le 0$.  In this setting, \eqref{eq:alphaTau} is hardest to prove for the smallest possible value of $q\ge n^{-1/r}$.  At this value of $q=n^{-1/r}$, \eqref{eq:alphaTau} reduces to showing
    \[0\le r(e(\nu)-1)+1-f(\nu)+r(|\nu|-2)-r(e(\nu)-1)=1-f(\nu)+r(|\nu|-2),\]
    and this inequality holds by \Cref{lem:fBounds}, proving \eqref{eq:alphaTau} when $r(e(\nu)-1)+1-f(\nu)\le 0$.

    Now consider the subcase that $r(e(\nu)-1)+1-f(\nu)> 0$, i.e. that $f(\nu)-1<r(e(\nu)-1)$.  Since we are only considering $\nu$ with $\del(F[\nu])\ge 1$, we see that every $\nu$ which we are considering lies in $A_F$.  As such, to prove \eqref{eq:alphaTau} in the case $r(e(\nu)-1)+1-f(\nu)>0$, it suffices to have
    \[q\le \min_{\nu\in A_F} n^{\frac{|\nu|-2-(e(\nu)-1)}{r(e(\nu)-1)+1-f(\nu)}}=\min_{\nu\in A_F} n^{\al(\nu)/r-1/r}=n^{\al(F)/r-1/r},\]
    where this first equality used that $\al(\nu)-1=\frac{r(|\nu|-2)-r(e(\nu)-1)}{r(e(\nu)-1)+1-f(\nu)}$ and the last equality used the definition of $\al(F)$.  This bound holds by our assumption on $q$ in this case, proving the result.

    Using similar logic for (c), we see that having $\tau(n,qn^2)\le n^{2-1/m_2(F)}$ is equivalent to having for all $\nu\sub V(F)$ with $\del(F[\nu])\ge 1$ and $e(\nu)\ge 2$ that
    \begin{equation}q^{e(\nu)-f(\nu)}\le n^{|\nu|-2-(e(\nu)-1)/m_2(F)}.\label{eq:betaTau}\end{equation}
    If $\nu$ is such that $f(\nu)=e(\nu)$ then \eqref{eq:betaTau} is satisfied precisely if $|\nu|-2\ge (e(\nu)-1)/m_2(F)$, which holds by definition of $m_2(F)$.  Otherwise in view of \Cref{lem:fvse}, we must have $f(\nu)>e(\nu)$ and hence $\nu\in B_F$.  In this case, showing \eqref{eq:betaTau} for all such $\nu$ is equivalent to showing
    \[q\ge \max_{\nu\in B_F} n^{\frac{|\nu|-2-(e(\nu)-1)/m_2(F)}{e(\nu)-f(\nu)}}=\max_{\nu\in B_F} n^{-\be(\nu)}=n^{-b(F)},\]
    which holds by assumption of the case that we are in.
\end{proof}

We now use these results to prove our main technical upper bounds on $\ex(G_{n,p},F)$, starting with our result for $\al(F)$.

\begin{proof}[Proof of \Cref{thm:alpha}]
    Recall that we wish to show that if $F$ is $r$-semi-bounded and contains a cycle, then there exists $C>0$ such that for all $n^{-\al(F)}\log^C(n)\le p\le 1$ for some large constant $C$ to be determined later, then a.a.s.\ we have
    \[\ex(G_{n,p},F)=O(p^{1-1/r}n^{2-1/r}).\]
    We note for later that $F$ containing a cycle implies $r\ge 2$.

    Letting $\tau$ be as in \Cref{cor:tauFormulation}, we have from this result that $F$ is $(C'n^{2-1/r},C',\tau)$-balanced for some constant $C'$.  Now define the function \[\tau'(n,qn^2)=\max\{n^{2-1/r-(r-1)\al(F)/r},q^{1-r}n\},\]
    which equivalently has $\tau'(n,qn^2)= q^{1-r}n$ for $q\le n^{\al(F)/r-1/r}$ and $\tau'(n,qn^2)=n^{2-1/r-(r-1)\al(F)/r}$ otherwise.  By \Cref{lem:tauProperties}(a) and (b), we have $\tau(n,qn^2)\le \tau'(n,qn^2)$ for all $q$.  As such, $F$ is also $(C' n^{2-1/r},C',\tau')$-balanced.    We can thus apply \Cref{Lemma:General Random Turan for Large p} with $q_0=C' n^{-1/r}$ to conclude that there exists some $C''>0$ such that a.a.s.\ for all $q\ge C' n^{-1/r}$ and $p$ with $pqn^2\to \infty$ we have
    \[\ex(G_{n,p},F)\le \max\{C'' pqn^2,\ n^{2-1/r-(r-1)\al(F)/r}(\log n)^{C''},\ q^{1-r}n,\ p^{1-1/r}n^{2-1/r}\}.\]
    We aim to apply this bound with $q=C'p^{-1/r}n^{-1/r}$ in the range $p\ge n^{-\al(F)}\log^C(n)$.  Note that for these parameters we have $q\ge q_0=C' n^{-1/r}$ and that $p q n^2\ge p n^{3/2}\to \infty$ since $r\ge 2$ and $p\ge n^{-\al(F)}\ge n^{-1}$ by \Cref{lem:alBound}, so this bound is indeed valid for these choice of parameters.  Plugging in these values gives
    \[\ex(G_{n,p},F)\le \max\{ C'C'' p^{1-1/r}n^{2-1/r},n^{2-1/r-(r-1)\al(F)/r}(\log n)^{C''},(C')^{1-r} p^{1-1/r} n^{2-1/r},p^{1-1/r}n^{2-1/r}\}.\]
    This gives the desired upper bound whenever $p^{1-1/r} n^{2-1/r}\ge n^{2-1/r-(r-1)\al(F)/r}(\log n)^{C''}$, and this precisely holds whenever $p\ge n^{-\al(F)}(\log n)^C$ for some large constant $C$, proving the result.
\end{proof}

For our result on $\be(F)$, we will need the following basic fact.

\begin{lem}\label{lem:m2}
    Let $F$ be a graph with $e(F)\ge 2$.  If $F$ has a cycle then $m_2(F)>1$.  If $F$ is a forest then $m_2(F)\le 1$ with equality whenever $F$ is not a subgraph of a matching.
\end{lem}
In fact we only need the weaker bound $m_2(F)>1/2$ when $F$ has a cycle to prove \Cref{thm:beta}, though some of our later results will require the full statement of this lemma.
\begin{proof}
    If $\nu\sub V(F)$ is the vertex set of a cycle of $F$ then $e(\nu)\ge |\nu|$, and hence
    \[m_2(F)\ge \frac{e(\nu)-1}{|\nu|-2}>1.\]
    On the other hand, if $F$ does not contain a cycle then every $\nu\sub V(F)$ satisfies $e(\nu)\le |\nu|-1$, showing $m_2(F)\le 1$.  If $F$ is not a subgraph of matching, then there exists some $\nu\sub V(F)$ inducing a path of length 2 which shows $m_2(F)\ge m_2(\nu)=1$, proving the result.
\end{proof}

We now mimic our argument for \Cref{thm:alpha} to prove our corresponding result for $\be(F)$, though in this case it will be somewhat more convenient to work with \Cref{Lemma:General Random Turan} over \Cref{Lemma:General Random Turan for Large p}.

\begin{proof}[Proof of \Cref{thm:beta}]
    We aim to show that if $F$ is $r$-semi-bounded and contains a cycle, then there exists $C>0$ such that for all $n^{-1/m_2(F)}\le p\le \min\{n^{\be(F)-1/m_2(F)},n^{1/r-1/m_2(F)}$\} we have a.a.s.
    \[\ex(G_{n,p},F)\le n^{2-1/m_2(F)}(\log n)^C.\]
    
    By \Cref{cor:tauFormulation}, we have that $F$ is $(C'n^{2-1/r},C',\tau)$-balanced for some $C'>0$ and $\tau$ as in \Cref{cor:tauFormulation}.   By \Cref{lem:tauProperties}(a), we can apply \Cref{Lemma:General Random Turan} to $F$ to conclude that there exists some $C''$ such that for any $q$ with    $qn^2\ge M(n)=C' n^{2-1/r}$ and any $p$ with $pqn^2\to \infty$, we have a.a.s.
    \begin{equation*}\ex(G_{n,p},F)\le \max\{C'' p qn^2,\tau(n,qn^2)(\log n)^{C''}\}\label{eq:betaUpper}\end{equation*}
   We will only apply this result at $q=C'p^{-1} n^{-1/m_2(F)}$, which does indeed satisfy $qn^2\ge C' n^{2-1/r}$ since $p\le n^{1/r-1/m_2(F)}$, and also that $pqn^2=C'n^{2-1/m_2(F)}\to \infty$ by \Cref{lem:m2}.  Because $ p\le n^{\be(F)-1/m_2(F)}$ we have $q\ge C' n^{-\be(F)}$, so by \Cref{lem:tauProperties}(c), the inequality above applied with $q= C' p^{-1} n^{-1/m_2(F)}$ gives
   \[\ex(G_{n,p},F)\le \max\{C'C'' n^{2-1/m_2(F)},n^{2-1/m_2(F)}(\log n)^{C''}\},\]
   proving the result.
\end{proof}

\subsection{Standard Random Tur\'an Tools}
The results in the previous subsection will be enough to establish our upper bounds for $\ex(G_{n,p},F)$ for large values of $p$.  Here we collect some known results to handle the remaining bounds, the first of which is the following.

\begin{prop}[\cite{nie2024random} Proposition 2.5]\label{prop:nieSpiroEasyRanges}
    Let $F$ be a graph with $\Delta(F)\ge 2$.  If $n^{-2}\ll p \ll n^{-\frac{1}{m_2(F)}}$, then a.a.s.
    \[\ex(G_{n,p},F)=(1+o(1))p {n\choose 2}.\]
    If $p\gg  n^{-\frac{1}{m_2(F)}}$, then a.a.s.
    \[\ex(G_{n,p},F)\ge  n^{2-\frac{1}{m_2(F)}}(\log n)^{-1}.\]
\end{prop}

It only remains to establish a lower bound on $\ex(G_{n,p},F)$ for $p$ large, and for this we use a known result lower bounding $\ex(G_{n,p},K_{r,t})$.  A proof for this result was sketched out by Morris and Saxton~\cite{morris2016number}, and this result can be formally derived as a consequence of\footnote{We note that the journal version of \cite[Proposition 2.2]{spiro2024random}  does not include the hypothesis $p\gg n^{-1} \log n$, but this is needed for an implicit application of a Chernoff bound in its proof to go through.} \cite[Proposition 2.2]{spiro2024random} together with some basic properties of $G_{n,p}$. 

\begin{prop}[\cite{morris2016number,spiro2024random}]\label{prop:KrtLower}
    If $\ex(n,K_{r,t})=\Theta(n^{2-1/r})$, then for all $p\gg n^{-1}\log n$ we have a.a.s.\ 
    \[\ex(G_{n,p},K_{r,t})=\Om(p^{1-1/r}n^{2-1/r}).\]
\end{prop}

\subsection{Proof of Main Results}
\subsubsection{Proof of \Cref{thm:maxDegree}}

Our first main result follows quickly from \Cref{thm:alpha} and \Cref{prop:alBound}.

\begin{proof}[Proof of \Cref{thm:maxDegree}]
    Recall that we wish to show that if $F$ has a bipartition $S\cup T$ and a vertex $v^*\in T$ such that every vertex in $T\sm \{v^*\}$ has degree at most $r$ and every vertex $u\in S$ has $|N_F(u)\cup \{v^*\}|\le \Del$; then there exists some $C$ such that for all $p\ge n^{-\frac{r-1}{r\Del-1}}(\log n)^C$, we have a.a.s.
    \[\ex(G_{n,p},F)=O(p^{1-1/r}n^{2-1/r}),\]
    with this bound moreover being tight whenever $F$ contains a $K_{r,t}$ with $\ex(n,K_{r,t})=\Theta(n^{2-1/r})$.  
    
    The lower bound follows from \Cref{prop:KrtLower} and that $n^{-\frac{r-1}{r\Del-1}}\gg n^{-1} \log n$.  Similarly, the upper bound trivially holds if $F$ is a forest since $p\gg n^{-\frac{r-1}{r\Del-1}}\gg n^{-1}$ implies
    \[\ex(G_{n,p},F)\le \ex(n,F)=O(n)=O(p^{1-1/r}n^{2-1/r}).\]
    It thus remains to prove the result in the case that $F$ contains a cycle.

    Define the graph $F'\supseteq F$ by adding any missing edges between $v^*$ and $S$.  Observe that $F'$ is $r$-semi-bounded and that every vertex in $S$ has degree at most $\Del$ in $F'$ by hypothesis.  The desired bound then follows from \Cref{prop:alBound} and \Cref{thm:alpha}.
\end{proof}

\subsubsection{Proof of \Cref{thm:upperSmallp}}
All our remaining proofs will in part require us to study the 2-density of graphs. To aid with this, given a graph $F$ and a subset $\nu\sub V(F)$ on at least 3 vertices, we define
\[m_2(\nu):=\frac{e(\nu)-1}{|\nu|-2},\]
noting that $m_2(F)$ is by definition the maximum of $m_2(\nu)$ over all $\nu$ with at least 3 vertices.  The following observation will be useful for studying the graphs of interest to us.

\begin{lem}\label{lem:includeComplete}
    Let $F$ be a bipartite graph on $S\cup T$ such that $T$ contains a vertex $v^*\in T$ which is adjacent to every vertex of $S$.  If $F$ contains a cycle, then any $\nu\sub V(F)$ with $|\nu|\ge 3$ and $m_2(\nu)=m_2(F)$ has $v^*\in \nu$.
\end{lem}
\begin{proof}
    Assume for contradiction that there exists some $\nu$ with $|\nu|\ge 3$ and $m_2(\nu)=m_2(F)$ such that $v^*\notin \nu$ and let $\nu^*:=\nu\cup \{v^*\}$.  We will show that $m_2(\nu^*)>m_2(\nu)$, contradicting $m_2(\nu)=m_2(F)=\max_{\nu'} m_2(\nu')$.  To this end, we have by definition 
    \begin{equation}m_2(\nu^*)-m_2(\nu)=\frac{e(\nu)+|S\cap \nu|-1}{|\nu|-1}-\frac{e(\nu)-1}{|\nu|-2}=\frac{(|\nu|-2)|S\cap \nu|-e(\nu)+1}{(|\nu|-1)(|\nu|-2)}.\label{eq:m2Star}
    \end{equation}
        We will get our desired contradiction if we can show that the quantity above is positive, and in particular it suffices to show that the numerator is positive.  To this end, observe that trivially $e(\nu)\le |S\cap \nu|(|\nu|-|S\cap \nu|)$, implying that
        \[(|\nu|-2)|S\cap \nu|-e(\nu)+1\ge (|S\cap \nu|-2)|S\cap \nu|+1.\]
        If $|S\cap \nu|\ge 2$, then the quantity above (and hence the difference in \eqref{eq:m2Star}) will be positive, giving the desired contradiction.  On the other hand, if $|S\cap \nu|\le 1$, then $F[\nu]$ is a forest and hence $m_2(\nu)\le 1$.  But $F$ containing a cycle implies that $m_2(F)>1$ by \Cref{lem:m2}, a contradiction to $m_2(\nu)=m_2(F)$.
\end{proof}

We now prove our second main result.

\begin{proof}[Proof of \Cref{thm:upperSmallp}]
    Recall that we wish to show that if $F$ is a graph with $e(F)\ge 2$ which has a bipartition $S\cup T$ such that there exists a vertex $v^*\in T$ which is adjacent to every vertex in $S$, and which is moreover such that any set $\nu\sub V(F)$ with $|\nu|\ge 3$ and $m_2(\nu)=m_2(F)$ has $S\sub \nu$; then a.a.s.\ 
    \[\ex(G_{n,p},F)=n^{2-1/m_2(F)}(\log n)^{\Theta(1)} \ \textrm{for all } n^{\frac{-1}{m_2(F)}}\ll p \le n^{\frac{1}{e(F)^2}-\frac{1}{m_2(F)}}.\]
    The lower bound for all $p\gg n^{-1/m_2(F)}$ follows from \Cref{prop:nieSpiroEasyRanges}, so it suffices to prove the upper bound.

    If $F$ is a forest, then $F$ is not a subgraph of a matching since $v^*$ is adjacent to every vertex of $S$ and $e(F)\ge 2$.  By \Cref{lem:m2} we have $m_2(F)=1$, and hence we trivially have for all values of $p$ that \[\ex(G_{n,p},F)\le \ex(n,F)=O(n)=O(n^{2-1/m_2(F)}),\]
    giving the desired result.  From now on we assume that $F$ contains a cycle.   Moreover, the fact that $F$ has a vertex complete to one side implies that $F$ is $e(F)$-semi-bounded with respect to $(S,T,v^*)$ since each $v\in T\sm \{v^*\}$ is trivially incident to at most $e(F)$ edges.

    In view of the observations above and \Cref{thm:beta}, we see that it suffice to show that $\be(F)\ge \frac{1}{e(F)^2}$, i.e.\ that every $\nu\in B_F$ satisfies $\be(\nu)\ge \frac{1}{e(F)^2}$.  For this we use the following.
    \begin{claim}
        Every $\nu\in B_F$ has $|\nu|\ge 3$ and $m_2(\nu)<m_2(F)$.
    \end{claim}
    \begin{proof}
        For the first part, assume for contradiction that there existed some $\nu\in B_F$ on at most 2 vertices.  Because $\nu\in B_F$ implies $\del(F[\nu])\ge 1$, we must have that $\nu$ consists of a single edge of $F$.  In this case $f(\nu)=1=e(\nu)$ by \Cref{lem:remove}(a), a contradiction to $\nu\in B_F$ which requires $f(\nu)>e(\nu)$.
        
        For the second half of the claim, we observe that \Cref{lem:includeComplete} and the hypothesis of the theorem implies any $\nu$ with $|\nu|\ge 3$ and $m_2(\nu)=m_2(F)$ has $S\cup \{v^*\}\sub \nu$.  By \Cref{lem:fvse} this implies that every such $\nu$ has $f(\nu)=e(\nu)$, and hence $\nu \notin B_F$ by definition.  We conclude that every $\nu\in B_F$ must have $m_2(\nu)<m_2(F)$.
    \end{proof}
    Now consider any $\nu\in B_F$, which by the claim above must have $|\nu|\ge 3$ and $m_2(\nu)<m_2(F)$.  Let $\nu'\sub V(F)$ be any set achieving $m_2(\nu')=m_2(F)$, and observe that 
        \[(e(\nu')-1)(|\nu|-2)-(e(\nu)-1)(|\nu'|-2)\ge 1,\]
        as this expression must be a positive integer in order to have $m_2(\nu)<m_2(\nu')=m_2(F)$.  Writing the definition of $\be(\nu)$ in terms of $m_2(\nu')=m_2(F)$ and using the inequality above together with the fact that $\nu\in B_F$ implies $f(\nu)-e(\nu)> 0$ gives that
        \[\be(\nu)=\frac{(e(\nu')-1)(|\nu|-2)-(e(\nu)-1)(|\nu'|-2|)}{(e(\nu')-1)(f(\nu)-e(\nu))}\ge \frac{1}{(e(\nu')-1)(f(\nu)-e(\nu))}\ge \frac{1}{e(F)^2},\]
        where this last step used $f(\nu)\le e(F)$ from \Cref{lem:fBounds}.
        We conclude that $\be(F)\ge \frac{1}{e(F)^2}$, which together with \Cref{thm:beta} gives the desired result.
\end{proof}

\subsubsection{Proof of Theorems~\ref{thm:multigraph} and \ref{thm:generalTightExamples}}

\begin{proof}[Proof of \Cref{thm:generalTightExamples}]
    We recall the setup of the theorem: $F$ is a graph which has a bipartition $S\cup T$ and a vertex $v^*\in T$ such that $v^*$ is adjacent to all of $S$ and such that every $v\in T\sm \{v^*\}$ has degree exactly $r\ge 2$.  For each $\mu\sub S$, let $N(\mu)\sub T$ denote the set of vertices of $T$ that are adjacent to at least one vertex of $\mu$.  
    
    We aim to show that if $F$ contains a subgraph isomorphic to some $K_{r,t}$ with $\ex(n,K_{r,t})=\Theta(n^{2-1/r})$, if $F$ is 2-balanced, and if $F$ satisfies
    \[\max_{\mu\sub S,\ |\mu|\ge 2} \frac{e(F[\mu\cup N(\mu)])-|N(\mu)|}{|\mu|-1}=\frac{e(F)-|T|}{|S|-1},\]
    then a.a.s.
    \[\ex(G_{n,p},F)=\begin{cases}
        \Theta(p^{1-1/r}n^{2-1/r}) &  n^{-\frac{|S|-1}{|S|-1+r(|T|-1)}}(\log n)^{O(1)}\le p,\\ 
        n^{2-\frac{|S|+|T|-2}{|S|-1+r(|T|-1)}}(\log n)^{\Theta(1)} & n^{-\frac{|S|+|T|-2}{|S|-1+r(|T|-1)}}\ll p\le n^{-\frac{|S|-1}{|S|-1+r(|T|-1)}}(\log n)^{O(1)},\\ 
        (1+o(1))p{n\choose 2} & n^{-2}\ll p\ll n^{-\frac{|S|+|T|-2}{|S|-1+r(|T|-1)}}.
    \end{cases}\]
        
    Observe that the hypothesis of $F$ being 2-balanced and the conditions on $F$ implies
    \[m_2(F)=\frac{e(F)-1}{v(F)-2}=\frac{|S|-1+r(|T|-1)}{|S|+|T|-2}.\]
    Thus \Cref{prop:nieSpiroEasyRanges} gives the desired bounds for $p\ll n^{-\frac{|S|+|T|-2}{|S|-1+r(|T|-1)}}$ and the desired lower bounds for $p\ll n^{-\frac{|S|-1}{|S|-1+r(|T|-1)}}$.  Similarly the hypothesis of the theorem and \Cref{prop:KrtLower} implies the lower bound for $p\gg n^{-\frac{|S|-1}{|S|-1+r(|T|-1)}}$.

    With the above in mind, all that remains is to prove the upper bounds for $p\gg n^{-\frac{|S|+|T|-2}{|S|-1+r(|T|-1)}}$.  For this we, claim that it suffices to prove that we have $\ex(G_{n,p},F)=O(p^{1-1/r}n^{2-1/r})$ a.a.s.\ for $p\gg n^{-\frac{|S|-1}{|S|-1+r(|T|-1)}}(\log n)^{C}$ for some $C$. 
    Indeed, in this case for any $p\le p_0:=n^{-\frac{|S|-1}{|S|-1+r(|T|-1)}}(\log n)^C$ we would have by monotonicity that a.a.s.
    \begin{align*}\ex(G_{n,p},F)\le \ex(G_{n,p_0},F) &\le n^{2-1/r-\frac{(r-1)(|S|-1)}{r[|S|-1+r(|T|-1)]}}(\log n)^{O(1)}\\ &=n^{2-\frac{r(|S|-1)+r(|T|-1)}{r[|S|-1+r(|T|-1)]}}(\log n)^{O(1)}\\ &=n^{2-\frac{|S|+|T|-2}{|S|-1+r(|T|-1)}}(\log n)^{O(1)}.\end{align*}

    Observe that $F$ is $r$-semi-bounded and contains a cycle due to it containing some $K_{r,t}$ with $r\ge 2$ and $\ex(n,K_{r,t})=\Theta(n^{2-1/r})$.  Thus in view of \Cref{thm:alpha}, to prove the desired upper bound on $\ex(G_{n,p},F)$ for large $p$, it suffices to show $\al(F)\ge \frac{|S|-1}{|S|-1+r(|T|-1)}$.

    To this end, let $\nu\in A_F$ be an element with $\al(\nu)=\al(F)$, and amongst such $\nu$ we further choose one with $|\nu|$ as large as possible.    Because every $v\in T\sm \{v^*\}$ has degree $r$, we have by \Cref{lem:simple} that every $\nu'\in A_F$ has $f(\nu')=|S\cap \nu'|+r(|T\cap \nu'|-1)$ and
    \begin{equation}\al(\nu')=\frac{(r-1)(|S\cap \nu'|-1)}{r(e(\nu')-|T\cap \nu'|)+1-|S\cap \nu'|}.\label{eq:easyAlFormula}\end{equation}
    We will prove the result by analyzing $\nu$ in terms of $\mu:=S\cap \nu$.  
    \begin{claim}
        We have $|\mu|\ge 2$.
    \end{claim}
    \begin{proof}
        If this were not the case, then the condition $\del(F[\nu])\ge 1$ given by having $\nu\in A_F$ would imply that $F[\nu]$ is a star with $e(\nu)=|\nu|-1$.  On the other hand, by the formula for $f$ mentioned above we would have \[f(\nu)-1=r(|\nu|-2)=r(e(\nu)-1),\]
        a contradiction to $\nu\in A_F$.
    \end{proof}
    \begin{claim}
        We have $T\cap \nu=N(\mu)$.
    \end{claim}
    \begin{proof}
        First observe that if $v\in T\cap \nu\sm N(\mu)$, then by definition of $\mu$ and $N(\mu)$; the vertex $v$ must have degree 0 in $F[\nu]$, a contradiction to $\nu\in A_F$ inducing a graph of minimum degree at least 1.  We conclude that $T\cap \nu$ is a subset of $N(\mu)$, which itself is a subset of $T$ by definition.  Now assume for contradiction that there existed some $v\in N(\mu)\sm \nu$ and consider $\nu'=\nu\cup \{v\}$.  
        
        We begin by showing $\nu'\in A_F$.  Indeed, observe that $e(\nu')\ge e(\nu)+1$ and that $f(\nu')=f(\nu)+r$.  Because $f(\nu)-1<r(e(\nu)-1)$ by assumption of $\nu\in A_F$, we also have $f(\nu')-1<r(e(\nu')-1)$.  Similarly $\del(F[\nu])\ge 1$ implies $\del(F[\nu'])\ge 1$, so  we conclude that $\nu'\in A_F$.

        By \eqref{eq:easyAlFormula}, we see that
        \[\al(\nu')\le \frac{(r-1)(|S\cap \nu|-1)}{r(e(\nu)+1-(|T\cap \nu|+1))+1-|S\cap \nu|}=\al(\nu),\]
        a contradiction to our choice of $\nu$ being the largest set in $A_F$ which minimizes $\al$.  We conclude that $v\in \nu$, giving the result.
    \end{proof}
    With these two claims, we can use \eqref{eq:easyAlFormula} to conclude
    \[\al(F)=\al(\nu)=\frac{(r-1)(|\mu|-1)}{r(e(\mu\cup N(\mu))-|N(\mu)|)+1-|\mu|}=\frac{r-1}{\frac{r(e(F[\mu\cup N(\mu)])-|N(\mu)|}{|\mu|-1}-1}\ge \frac{r-1}{\frac{r(e(F)-|T|)}{|S|-1}-1},\]
    where the inequality used the hypothesis from the theorem.  Rearranging this expression and using that $e(F)=|S|+r(|T|-1)$ by definition of $F$, we find that
    \[\al(F)\ge \frac{(r-1)(|S|-1)}{r(r-1)|T|+(r-1)|S|-r^2+1}=\frac{|S|-1}{|S|-1+r|T|-r},\]
    proving the result.      
\end{proof}
As an aside, one can more generally use \Cref{thm:alpha} to prove the upper bounds predicted in \Cref{conj:MS} for $r$-semi-bounded graphs $F$ whenever $\ex(n,F)=\Theta(n^{2-1/r})$ and $\al(F)=\frac{r}{(r-1)m_2(F)}-\frac{1}{r-1}$.  Indeed, the hypothesis given in \Cref{thm:generalTightExamples} is designed specifically to achieve this end.

It remains to prove \Cref{thm:multigraph}, which we recall involves graphs $F_M$ obtained from a multigraph $M$ by subdividing each edge of $M$ and then adding a vertex adjacent to every vertex of $M$.  This result will essentially follow from \Cref{thm:generalTightExamples} after verifying that $F_M$ is 2-balanced whenever it satisfies the hypothesis of \Cref{thm:multigraph}.

\begin{lem}\label{lem:multiBalanced}
    If $M$ is a multigraph 
     satisfying $e(M)\ge 1$ and
    \[\max_{\mu\sub V(M),\ |\mu|\ge 2}\frac{e(M[\mu])}{|\mu|-1}=\frac{e(M)}{v(M)-1},\]
    then
    \[m_2(F_M)=\frac{e(F_M)-1}{v(F_M)-2}=\frac{2e(M)+v(M)-1}{e(M)+v(M)-1}.\]
\end{lem}
\begin{proof}
    Let $\nu\sub V(F_M)$ be such that $m_2(\nu)=m_2(F_M)$.  By \Cref{lem:includeComplete} we necessarily have $v^*\in \nu$ since $e(M)\ge 1$ implies $F_M$ contains a 4-cycle.  We will need a slight strengthening of this observation.  To this end, let $\mu:=V(M)\cap \nu$.
    \begin{claim}
        The set $\nu\sm V(M)$ consists precisely of $v^*$ together with every $v\in V(F_M)\sm V(M)$ which is adjacent to two vertices $u_1,u_2\in \mu$.
    \end{claim}
    \begin{proof}
        That $v^*\in \nu\sm V(M)$ follows from our comment above. Assume for contradiction that there exists some $v\in V(F_M)\sm (V(M)\cup \{v^*\})$ which is adjacent to some $u_1,u_2\in \mu$ with $v\notin \nu$, and let $\nu'=\nu\cup \{v\}$.  Since $v\ne v^*$, the set $\nu'$ induces exactly 2 more edges than $\nu$.  This implies
        \begin{equation}m_2(\nu')-m_2(\nu)=\frac{e(\nu)+1}{|\nu|-1}-\frac{e(\nu)-1}{|\nu|-2}=\frac{2|\nu|-e(\nu)-3}{(|\nu|-1)(|\nu|-2)}.\label{eq:munu}\end{equation}
        Observe that if $m:=|\mu|=| V(M)\cap \nu|$ then $e(\nu)\le m+2(|\nu|-m-1)=2|\nu|-m-2$ since $m$ edges are incident to $v^*$ and at most 2 edges are incident to each vertex in $\nu \sm (V(M)\cup \{v^*\})$.  Since $u_1,u_2\in \mu$ we have $m\ge 2$, in total giving $e(\nu)\le 2|\nu|-4$.  This together with \eqref{eq:munu} implies $m_2(\nu')>m_2(\nu)$, a contradiction to $m_2(\nu)=m_2(F_M)$.

        Now assume for contradiction that $\nu$ contained some $v\in V(F_M)\sm (M\cup \{v^*\})$ which was adjacent to at most 1 vertex of $\mu$ and let $\nu':=\nu\sm \{v\}$.  First observe that $m_2(\nu)=m_2(F_M)>1$ since $F_M$ contains a cycle.  Having $m_2(\nu)>1$ implies $e(\nu)\ge |\nu|$, which in turn implies that $F[\nu]$ contains an (even) cycle, and hence $F[\nu']$ contains at least 2 edges is a set of size at least 3 (meaning $m_2(\nu')$ is defined) and satisfies
        \[m_2(\nu')-m_2(\nu)=\frac{e(\nu)-2}{|\nu|-3}-\frac{e(\nu)-1}{|\nu|-2}=\frac{e(\nu)-|\nu|+1}{(|\nu|-3)(|\nu|-2)}>0.\]
        This again gives a contradiction to $m_2(\nu)=m_2(F)$, giving the claim.
    \end{proof}
    By the claim, we find
    \[m_2(F)=m_2(\nu)=\frac{2e(M[\mu])+|\mu|-1}{e(M[\mu])+|\mu|-1}=1+\frac{e(M[\mu])}{e(M[\mu])+|\mu|-1}.\]
    We claim that this expression is maximized when $\mu=V(M)$, which will give the desired result.  Indeed, by taking the recipricol this is equivalent to saying that
    \[\frac{e(M)+v(M)-1}{e(M)}=\min_{\mu'\sub V(M)} \frac{e(M[\mu'])+|\mu'|-1}{e(M[\mu'])}=1+\min_{\mu'\sub V(M)} \frac{|\mu'|-1}{e(M[\mu'])},\]
    and this holds by the hypothesis of the lemma.  
\end{proof}
We now finish the proof of \Cref{thm:multigraph}.
\begin{proof}[Proof of \Cref{thm:multigraph}]
    Let $M$ be a multigraph as in the hypothesis of the theorem.  Observe that $F_M$ has a bipartition $S\cup T$ with $S=V(M)$ and $T=V(F_M)\sm V(M)$ such that every vertex in $T\sm \{v^*\}$ has degree 2, and such that $F$ contains a $K_{2,2}$ (since $e(M)\ge 1$), which is well known to satisfy $\ex(n,K_{2,2})=\Theta(n^{3/2})$.  \Cref{lem:multiBalanced} verifies that $F_M$ is 2-balanced, so in view of \Cref{thm:generalTightExamples} and the fact that $|S|=v(M)$ and $|T|=e(M)+1$, to prove the result it suffices to verify that
    \[\max_{\mu\sub S,\ |\mu|\ge 2} \frac{e(F[\mu\cup N(\mu)])-|N(\mu)|}{|\mu|-1}=\frac{e(F_M)-|T|}{|S|-1}=\frac{v(M)+2e(M)-e(M)-1}{v(M)-1}=1+\frac{e(M)}{v(M)-1},\]
    where here $N(\mu)$ is the set of vertices in $T$ which are adjacent to a vertex in $\mu$.  Observe that $N(\mu)$ always consists of $v^*$, the edges in $E(M[\mu])$ (i.e.\ the vertices in $T\sm \{v^*\}$ which have both their neighbors in $\mu$), and some remainder set $R_\mu:=N(\mu)\sm (\{v^*\}\cup E(M[\nu]))$ which consists precisely of the set of vertices adjacent to 1 vertex of $\mu$.  In this case we observe
    \begin{align*}\frac{e(F[\mu\cup N(\mu)])-|N(\mu)|}{|\mu|-1}&=\frac{|\mu|+2e(M[\mu])+|R_\mu|-1-e(M[\mu])-|R_\mu|}{|\mu|-1}\\ &=1+\frac{e(M[\mu])}{|\mu|-1}\le 1+\frac{e(M)}{v(M)-1},\end{align*}
    with this last inequality using the hypothesis of the theorem.  We conclude that we can apply \Cref{thm:generalTightExamples} to obtain the desired random Tur\'an bounds on $F_M$.
\end{proof}

\section{Concluding Remarks}\label{sec:concluding}
In this paper, we proved effective upper bounds on the random Tur\'an problem for bipartite graphs $F$ which have a vertex complete to one side of the bipartition through our technical results Theorems~\ref{thm:alpha} and \ref{thm:beta}, from which we were able to derive a number of related upper bounds.  In particular, we proved \Cref{thm:maxDegree} which can be thought of as a probabilistic analog of the bound $\ex(n,F)=O(n^{2-1/r})$ for bipartite graphs $F$ where one of the parts has all but 1 vertex of degree at most $r$.  In fact, the standard proof of this extremal result allows for up to $r$ vertices to have degree more than $r$.  It seems plausible that our results could extend to this setting as well, which would be a natural limit for what we would could prove with this approach.  We require some additional notation to state such a bound precisely.

\begin{defn}
    We say that a bipartite graph $F$ is a \textit{$(c,r)$-bounded} with respect to a triple of sets $(S,T,T^*)$ if $S\cup T$ is a bipartition of $F$, if $T^*\sub T$ is a set of $c$ vertices which are adjacent to every vertex of $S$, and if every vertex in $T\sm T^*$ has degree at most $r$.  We simply say that $F$ is $(c,r)$-bounded if it is $(c,r)$-bounded with respect to some triple $(S,T,T^*)$.

    Given a graph $F$ which is $(c,r)$-bounded with respect to a triple $(S,T,T^*)$, we define for each $\nu\sub V(F)$ the quantity
    \[e(\nu):=e(F[\nu]),\]
		and if $e(\nu)\ge 1$ we define
        \[f(\nu):=c|S\cap \nu|+\sum_{v\in T\cap \nu\sm T^*}\deg_F(v)-g(\nu),\]
        where
        \[g(\nu):=\begin{cases}
    |T^*\sm \nu| & T^*\cap \nu \ne \emptyset,\\ 
    \operatorname*{max} \limits_{\substack{w\in T\cap \nu \\ \deg_{F[\nu]}(w)\ge 1}} \deg_F(w)+(c-1)\deg_{F[\nu]}(w) &  T^*\cap \nu =\emptyset.
\end{cases}\]
We define $\al(F),\be(F)$ exactly as we did for $r$-semi-bounded graphs in terms of this new definition of $f$.
\end{defn}
Observe that $c=1$ exactly recovers the definitions we had before, and with some work one can show that analogs of the lemmas in \Cref{sec:preliminaries} continue to hold for $(c,r)$-bounded graphs with $1\le c\le r$ after some modifications to their statements.  With this in mind, we believe our results can be generalized in the following way.


\begin{conj}
    Theorems~\ref{thm:alpha} and \ref{thm:beta} hold for $(c,r)$-bounded graphs with $1\le c\le r$.
\end{conj}

An issue that one runs into by trying to mimic our current proof in this setting is that we now need to choose how to embed all of $T^*$ and $S$ such that they contain no saturated subgraphs in $G$.  When $|T^*|=1$ the only such subgraphs are $T$-stars, but for $|T^*|>1$ we need to handle more general complete bipartite graphs. In particular, avoiding saturated $K_{t', s'}$ with the $t'$-set a subset of $T^*$ becomes difficult. 


\bibliographystyle{abbrv}

\bibliography{refs}

\begin{thebibliography}{10}

\bibitem{alon2003turan}
N.~Alon, M.~Krivelevich, and B.~Sudakov.
\newblock Tur{\'a}n numbers of bipartite graphs and related {R}amsey-type questions.
\newblock {\em Combinatorics, Probability and Computing}, 12(5-6):477--494, 2003.

\bibitem{balogh2015independent}
J.~Balogh, R.~Morris, and W.~Samotij.
\newblock Independent sets in hypergraphs.
\newblock {\em Journal of the American Mathematical Society}, 28(3):669--709, 2015.

\bibitem{collares2016maximum}
M.~Collares and R.~Morris.
\newblock Maximum-size antichains in random set-systems.
\newblock {\em Random Structures \& Algorithms}, 49(2):308--321, 2016.

\bibitem{conlon2010approximate}
D.~Conlon, J.~Fox, and B.~Sudakov.
\newblock An approximate version of {S}idorenko’s conjecture.
\newblock {\em Geometric and Functional Analysis}, 20:1354--1366, 2010.

\bibitem{conlon2016combinatorial}
D.~Conlon and W.~T. Gowers.
\newblock Combinatorial theorems in sparse random sets.
\newblock {\em Annals of Mathematics}, pages 367--454, 2016.

\bibitem{conlon2018some}
D.~Conlon, J.~H. Kim, C.~Lee, and J.~Lee.
\newblock Some advances on {S}idorenko's conjecture.
\newblock {\em Journal of the London Mathematical Society}, 98(3):593--608, 2018.

\bibitem{conlon2017finite}
D.~Conlon and J.~Lee.
\newblock Finite reflection groups and graph norms.
\newblock {\em Advances in Mathematics}, 315:130--165, 2017.

\bibitem{conlon2021extremal}
D.~Conlon and J.~Lee.
\newblock On the extremal number of subdivisions.
\newblock {\em International Mathematics Research Notices}, 2021(12):9122--9145, 2021.

\bibitem{coregliano2021biregularity}
L.~N. Coregliano and A.~A. Razborov.
\newblock Biregularity in {S}idorenko's conjecture.
\newblock {\em arXiv preprint arXiv:2108.06599}, 2021.

\bibitem{fox2017local}
J.~Fox and F.~Wei.
\newblock On the local approach to {S}idorenko's conjecture.
\newblock {\em Electronic Notes in Discrete Mathematics}, 61:459--465, 2017.

\bibitem{furedi1991turan}
Z.~F{\"u}redi.
\newblock On a {T}ur{\'a}n type problem of {E}rd{\H{o}}s.
\newblock {\em Combinatorica}, 11(1):75--79, 1991.

\bibitem{hatami2010graph}
H.~Hatami.
\newblock Graph norms and {S}idorenko’s conjecture.
\newblock {\em Israel Journal of Mathematics}, 175:125--150, 2010.

\bibitem{haxell1995turan}
P.~E. Haxell, Y.~Kohayakawa, and T.~Luczak.
\newblock {T}ur{\'a}n' s extremal problem in random graphs: Forbidding even cycles.
\newblock {\em Journal of Combinatorial Theory, Series B}, 64(2):273--287, 1995.

\bibitem{jiang2022balanced}
T.~Jiang and S.~Longbrake.
\newblock Balanced supersaturation and {Tur}{\' a}n numbers in random graphs.
\newblock {\em Advances in Combinatorics}, jul 15 2024.

\bibitem{jiang2024number}
T.~Jiang and S.~Longbrake.
\newblock On the number of {$H$}-free hypergraphs.
\newblock {\em Forum of Math, Sigma}, page paper e20, February 2026.

\bibitem{kim2016two}
J.~H. Kim, C.~Lee, and J.~Lee.
\newblock Two approaches to {S}idorenko’s conjecture.
\newblock {\em Transactions of the American Mathematical Society}, 368(7):5057--5074, 2016.

\bibitem{kohayakawa1998extremal}
Y.~Kohayakawa, B.~Kreuter, and A.~Steger.
\newblock An extremal problem for random graphs and the number of graphs with large even-girth.
\newblock {\em Combinatorica}, 18(1):101--120, 1998.

\bibitem{li2011logarithimic}
J.~Li and B.~Szegedy.
\newblock On the logarithimic calculus and {S}idorenko's conjecture.
\newblock {\em arXiv preprint arXiv:1107.1153}, 2011.

\bibitem{lovasz2011subgraph}
L.~Lov{\'a}sz.
\newblock Subgraph densities in signed graphons and the local {S}imonovits--{S}idorenko conjecture.
\newblock {\em The Electronic Journal of Combinatorics}, pages P127--P127, 2011.

\bibitem{mckinley2023random}
G.~McKinley and S.~Spiro.
\newblock The random {T}ur\'an problem for theta graphs.
\newblock {\em arXiv preprint arXiv:2305.16550}, 2023.

\bibitem{morris2016number}
R.~Morris and D.~Saxton.
\newblock The number of ${C}_{2\ell}$-free graphs.
\newblock {\em Advances in Mathematics}, 298:534--580, 2016.

\bibitem{mubayi2023random}
D.~Mubayi and L.~Yepremyan.
\newblock On the random {T}ur\'an number of linear cycles.
\newblock {\em arXiv preprint arXiv:2304.15003}, 2023.

\bibitem{nie2023random}
J.~Nie.
\newblock Random {T}ur\'an theorem for expansions of spanning subgraphs of tight trees.
\newblock {\em arXiv preprint arXiv:2305.04193}, 2023.

\bibitem{nie2023tur}
J.~Nie.
\newblock {T}ur{\'a}n theorems for even cycles in random hypergraph.
\newblock {\em Journal of Combinatorial Theory, Series B}, 167:23--54, 2024.

\bibitem{nie202X}
J.~Nie and S.~Spiro.
\newblock Random {T}ur\'an problems for ${K}_{s,t}$ expansions.
\newblock {\em In Preparation}.

\bibitem{nie2023sidorenko}
J.~Nie and S.~Spiro.
\newblock {S}idorenko hypergraphs and random {T}ur\'an numbers.
\newblock {\em arXiv preprint arXiv:2309.12873}, 2023.

\bibitem{nie2024random}
J.~Nie and S.~Spiro.
\newblock Random {T}ur\'an problems for hypergraph expansions.
\newblock {\em arXiv preprint arXiv:2408.03406}, 2024.

\bibitem{saxton2015hypergraph}
D.~Saxton and A.~Thomason.
\newblock Hypergraph containers.
\newblock {\em Inventiones mathematicae}, 201(3):925--992, 2015.

\bibitem{schacht2016extremal}
M.~Schacht.
\newblock Extremal results for random discrete structures.
\newblock {\em Annals of Mathematics}, pages 333--365, 2016.

\bibitem{sidorenko1986extremal}
A.~Sidorenko.
\newblock Extremal problems in graph theory and inequalities in functional analysis.
\newblock In {\em Proc. Soviet Seminar on Discrete Math. Appl., Moscow Univ. Press}, pages 99--105, 1986.

\bibitem{sidorenko1991inequalities}
A.~F. Sidorenko.
\newblock Inequalities for functionals generated by bipartite graphs.
\newblock {\em Diskret. Mat.}, 3(3):50--65, 1991.

\bibitem{spiro2024random}
S.~Spiro.
\newblock Random polynomial graphs for random {T}ur{\'a}n problems.
\newblock {\em Journal of Graph Theory}, 105(2):192--208, 2024.

\bibitem{spiro2021relative}
S.~Spiro and J.~Verstra{\"e}te.
\newblock Relative {T}ur{\'a}n problems for uniform hypergraphs.
\newblock {\em SIAM Journal on Discrete Mathematics}, 35(3):2170--2191, 2021.

\bibitem{szegedy2014information}
B.~Szegedy.
\newblock An information theoretic approach to {S}idorenko's conjecture.
\newblock {\em arXiv preprint arXiv:1406.6738}, 2014.

\end{thebibliography}

\newpage

\appendix

\section{Appendix: Using Containers}

Here we prove \Cref{Lemma:General Random Turan for Large p}, the statement of which we recall below.  We emphasize that this proof is for the most part a straightforward generalization of Theorems 6.1 and 7.6 of \cite{morris2016number}, and as such our exposition will be somewhat terse is various places.  

\begin{prop}
Let $F$ be a graph.  If there exists a  $\alpha, \beta, C>0$ and positive functions $q_0(n)$ and $\tau=\tau(n,q)$ such that 
\begin{itemize}
    \item[(a)] $F$ is $(q_0(n)n^2,\gamma , \tau)$-balanced, and
    \item[(b)] For all sufficiently large $n$ and $q\ge q_0(n)$, we have $\tau(n,q) = C\max \{n^{2 - \alpha}, q^{1 - \beta}n \}. $
\end{itemize}
then there exists $C'\ge 0$ such that for all sufficiently large $n$, $q\ge q_0(n)$, and $0<p\le 1$ with $pqn^2\rightarrow\infty$ as $n\rightarrow \infty$, we have a.a.s.
$$
\ex(G_{n,p},F)\le \max\l\{C'pqn^2,n^{2- \alpha}(\log n)^{C'}, q^{1- \beta}n, p^{1 - \frac{1}{\beta}}n^{2 - \frac{1}{\beta}}\r\}.
$$
\end{prop}

Our starting point is a standard lemma from the method of hypergraph containers which was independently developed by Baloh, Morris, and Samotij (Theorem 2.2 in \cite{balogh2015independent}) as well as  Saxton and Thomason (Theorem 3.4 in \cite{saxton2015hypergraph}).  
Here and throughout $\P(S)$ denotes the power set of a set $S$, and given a hypergraph $\Hc$, we let $\Delta_i(\Hc) := \max_{J \subseteq V(\Hc), |J| = i} \deg_\Hc(J)$ and $\mathcal{I}(\c{H})$ denote the set of independent sets of $\c{H}$. 

\begin{lem}[Balogh-Morris-Samotij \cite{balogh2015independent} and Saxton-Thomason \cite{saxton2015hypergraph}]\label{containerlemma}
For every $t \in \NN$  there exists a $\delta > 0$ such that the following always holds. Let $\cH$ be a $t$-uniform hypergraph on $N$ vertices and $\gamma,\tau$ real numbers such that for all $1 \leq i \leq t$, $$\Delta_i(\cH) \leq \gamma \left( \frac{\tau}{v(\cH)}\right)^{i - 1} \frac{\gamma e(\cH)}{v(\cH)}.$$ Then there exists a collection $\C$ of subsets of $V(\cH)$ and function $\iota : \I(\cH) \rightarrow \P(V(\cH))$ and $h: \P(V(\cH)) \rightarrow \c{C}$ satisfying the following:
\begin{enumerate}
\item For every $I \in \I(\cH)$, $|S(I)| \leq (t - 1)\tau$ and $\iota(I) \subseteq I \subseteq h(\iota(I))$. 
\item For every $C \in \C$, $|V(C)| \leq v(\cH) - \frac{\delta}{\gamma}  v(\cH)$. 
\end{enumerate}
\end{lem}

We also need a purely arithmetical lemma due to Neto and Morris.

\begin{lem}[Lemma 4.3 in \cite{collares2016maximum}]\label{NetoMorris}
Let $M > 0$, $s > 0$, and $0 < \delta < 1$. For any finite sequence $a_1, \dots, a_m$ of real numbers summing to $s$ such that $1 \leq a_j \leq (1 - \delta)^j M$ for each $j \in [m]$, we have
$$s \log (s) \leq \sum_{j = 1}^m a_j \log a_j + O\left(\frac{M}{\delta^2}\right). $$
\end{lem}

Before getting into the meat of our proof, let us briefly sketch  the high-level idea of our argument.  Our goal will be to construct a set of container $\c{C}$, i.e.\ a set of $n$-vertex graphs such that every $n$-vertex $F$-free graph is a subgraph of some $G\in \c{C}$.  We moreover will want to impose that every element of $\c{C}$ has at most some given size $q_0n^2$.  To accomplish this, we initially start with $\c{C}=\{K_n\}$ which is trivially a set of containers.  Iteratively given some set of containers $\c{C}$, if there exists a graph $G\in \c{C}$ with more than $q_0n^2$ edges, then we apply the container lemma to $G$ (or more precisely, to the hypergraph which encodes copies of $F$ in $G$) which gives a set of graphs $\c{C}'$ each of size at most $(1-\del/\gam) e(G)$ such that $\c{C}'\cup \c{C}\sm \{G\}$ is a set of containers.  Repeating this process at most $O(\log n)$ times gives a set of containers each of which has size at most $q_0n^2$, and moreover the number of such containers can be shown to be relatively small by using \Cref{NetoMorris}.

The argument above is all that is needed to get our main result up to logarithmic factors, and to get rid of these factors we need to be a bit more careful in our analysis.  Specifically, each container $G\in \c{C}$ that we ultimately end up with comes about from a repeated sequence of applications of the container lemma, and hence to each container $G$ there is some sequence of fingerprints (i.e.\ images of the $S$ function from the container lemma) associated to it.  Ultimately we need to (iteratively) keep track of what this sequence of fingerprints is for each container, which will complicate our notation somewhat.

\newcommand{\Forb}{\mathrm{Forb}}

\begin{prop}\label{appendixproposition}
    Let $F$ be a graph and let $\Forb(n,F)$ denote the set of $n$-vertex $F$-free graphs.  If there exists a  $C>0, \beta > 1 > \alpha > 0$ and positive functions $M=M(n)$ and $\tau=\tau(n,q)$ such that 
\begin{itemize}
    \item[(a)] $F$ is $(M,\gamma , \tau)$-balanced, and
    \item[(b)] For all sufficiently large $n$ and $q\ge M$, we have $\tau(n,qn^2) = C\max \{n^{2 - \alpha}, q^{1 - \beta}n \}. $
\end{itemize}
then there exists a constant $C'>0$ such that for all $n$ and $q_0\ge M(n)$, there exists a set $\c{S}$ and functions $f:\c{S}\to \c{P}(E(K_n))$ and $g:\Forb(n,F)\to \c{S}$ with the following properties:
\begin{itemize}
    \item[(i)] Each element $\bold{S}\in \c{S}$ is a sequence of edge-disjoint subgraphs of $K_n$ such that $f(\bold{S})$ is edge-disjoint from each of the graphs $\bold{S}_i$.
    \item[(ii)] For each $H\in \Forb(n,F)$, we have
    \[\bigcup_i g(H)_i\sub H \sub f(g(H))\cup \bigcup_i g(H)_i.\]
    That is, every graph in the sequence $g(H)$ is a subgraph of $H$ and $H$ is contained in the union of these subgraphs and $f(g(H))$.
    \item[(iii)] We have $e(f(\bold{S}))\le q_0 n^2$ for all $\bold{S}\in \c{S}$
    \item[(iv)] For every integer $t$, the number of $\bold{S}\in \c{S}$ with $\sum_i e(\bold{S}_i)=t$ is at most 
    $$\left( \frac{C'n^{2 - \frac{1}{\beta}}}{t
    }\right)^{\left({\frac{\beta}{\beta - 1}}\right)t} \cdot  \exp(C'q_0^{1 - \beta}n+ C'\log(n)^2n^{ 2- \alpha} )$$
\end{itemize}
\end{prop}

\begin{proof}
    Our proof strategy will be to iteratively build a sequence of objects $\c{S}_j,g_j,f_j$ satisfying (i) and (ii) along with the further properties: 
    \begin{itemize}
        \item[(iii')] We have $e(f(\bold{S}))\le \max\{q_0 n^2,(1-\del/\gam)^j n^2\}$ for all $\bold{S}\in \c{S}$.
        \item[(iv')] Let $\bold{s} \in \NN^{\leq j}$, and let $\c{S}_j(\bold{s})$ denote the set of $\bold{S}$ with $e(\bold{S}_i) = \bold{s}_i$, then 

        $$|\c{S}_j(\bold{s})|\le \prod_{i = 1}^{|\bold{s}|} \left( \frac{C'n^{2 - \frac{1}{\beta}}}{\bold{s}_i} \right)^{\left({\frac{\beta}{\beta - 1}}\right)\bold{s}_i}\cdot \exp(C'j \log(n)n^{2 - \alpha}).$$
        
    \end{itemize}
    

    To this end, let $\c{S}_0=\emptyset$, define $g_0:\Forb(n,F)\to \c{S}_0$ to be the unique function of this form, and define $f_0:\c{S}_0\to \c{P}(K_n)$ by $f_0(\emptyset)=K_n$.  It is immediate to check that these objects satisfy (i), (ii), (iii'), and (iv').

    Iteratively assume that we have constructed $\c{S}_j,g_j,f_j$ satisfying (i), (ii), (iii'), and (iv').  To construct the next iteration, we split $\c{S}_j$ into two sets, namely $\S_{j}^{\text{good}} = \{ \bold{S} \in \S_j : |f_j(\bold{S})| \leq q_0 n^2 \}$ and $\S_{j}^{\text{bad}} = \{ \bold{S} \in \S_j : |f_j(\bold{S})| > qn^2 \}$.  For each $\bold{S}\in \S_j^{\text{bad}}$, let $\tau_{\bold{S}}=\tau(n,e(f_j(\bold{S})))$.  By hypothesis of $F$ and $q_0\ge M(n)$, for each $\bold{S}\in \S_{j}^{\text{bad}}$ there exists a non-empty collection $\c{H}_{\bold{S}}$ of copies of $F$ in $f_j(\bold{S})$ satisfying
    \[\Del_i(\c{H}_{\bold{S}})\le \frac{\gam e(\c{H}_{\bold{S}})}{e(f_j(\bold{S}))}\left(\frac{\tau_{\bold{S}}}{e(f_j(\bold{S}))} \right)^{i-1}.\]
    By the container lemma, there exists a collection $\c{C}_{\bold{S}}$ of subgraphs of $f_j(\bold{S})$ together with functions 
    $\iota_{\bold{S}}$ from $F$-free subgraphs of $f_j(\bold{S})$ to $ \binom{f_j(\bold{S})}{ \leq \tau_\bold{S}}$, as well as a function $h_{\bold{S}}: \binom{f_j(\bold{S})}{ \leq \tau_\bold{S}}\to \c{C}_{\bold{S}}$.

    We now define $\c{S}_{j+1},g_{j+1}$, and $f_{j+1}$ as follows.  We define $\c{S}_{j+1}$ to include all of $\c{S}_{j}^{\text{good}}$, and for each $\bold{S}\in \c{S}_{j}^{\text{bad}}$ we add to $\c{S}_{j+1}$ every sequence obtained by concatenating an element of $\binom{f_j(\bold{S})}{ \tau_\bold{S}}$ to the end of $\bold{S}$.  Note that these sequences are still edge disjoint since $f_j(\bold{S})$ is edge-disjoint from each $\bold{S}_i$ by hypothesis of (i).  For each $H\in \Forb(n,F)$, we define $g_{j+1}(H)=g_j(H)$ if $g_j(H)\in \c{S}_{j}^{\text{good}}$, otherwise $g_{j+1}(H)$ is defined to be $g_j(H)$ concatenated with $\iota_{\bold{S}}(H)$, which is well-defined because (ii) implies that $H \setminus g_j(H)$ is an $F$-free subgraph of $f_j(\bold{S})$.  Finally, we define $f_{j+1}(\bold{S}')=f_j(\bold{S}')$ if $\bold{S}'\in \c{S}_{j}^{\text{good}}$ and otherwise if $\bold{S}'$ consists of some $\bold{S}\in \S_j^{\text{bad}}$ and another set $S\in \binom{f_j(\bold{S})}{ \tau_\bold{S}}$, then we define $f_{j+1}(\bold{S}')=h_{\bold{S}}(S)\sm S$.

    \begin{claim}
        These definitions satisfy (i), (ii), (iii'), and (iv').
    \end{claim}
    \begin{proof}
        The first three conditions are relatively straightforward to verify after unwinding the definitions, so we focus our attention on (iv').  
        Observe that for a fixed $\bold{s} \in \NN^{\leq j + 1}$, the only one for whom $|\c{S}_{j + 1}(\bold{s})|> |\c{S}_{j}(\bold{s})|$ 
        is such that $|\bold{s}| = j + 1$. In this case, letting $\bold{s'}$ be the subsequence of $\bold{s}$ formed by the first $j$ coordinates
        \begin{equation*}
            |\S_{j + 1}(\bold{s})| \leq |\S_{j}(\bold{s}')| \cdot \sum_{0\le q\le 1: qn^2\in \N; \bold{s}_{j + 1} \leq \tau(n, qn^2)} \binom{qn^2}{\bold{s}_{j + 1}}
        \end{equation*}
        since every sequence in $\c{S}_{j+1}(\bold{s})$ is  comes from concatenating some $\bold{S}\in \c{S}_j(\bold{s}')$ with $\bold{s}_{j + 1}$ edges from the graph $G_{\bold{S}}$ which has some $qn^2$ edges.
    %

    First consider the case that the maximum above is achieved by some $q$ with $\tau(n,qn^2)=C n^{2-\al}$.  In this case we trivially have ${q n^2\choose s_j}\le (n^2)^{s_j}\le \exp(2C n^{2-\alpha} \log(n))$, which in total gives the desired bound.  Otherwise if $\tau(n,qn^2)=C q^{1-\beta} n$ then we have $\bold{s}_j\le C q^{1-\beta} n$ and hence $q\leq (\bold{s}_j/Cn)^{1/(1-\beta)}$, implying that
    \begin{align*}
       {q n^2\choose \bold{s}_j}&\le  \left( \frac{ e q n^2}{\bold{s}_j}\right)^{\bold{s}_j} \\
       &\leq \left( C' \frac{n^{2 - \frac{1}{1 - \beta}}}{\bold{s}_j^{1 - \frac{1}{1 - \beta}}}\right)^{\bold{s}_j}\\
       &=  \left( C'\frac{n^{2 - \frac{1}{\beta}}}{\bold{s}_j}\right)^{\frac{\beta}{\beta - 1}\bold{s}_j}
    \end{align*}
    This shows the construction satisfies (iv').

   \end{proof}
   By (iii'), for for some $j'=\Theta(\log n)$ the objects $\S_{j'},g_{j'},f_{j'}$ will satisfy condition (iii). By (iv') and (iii') these objects will further satisfy (iv). Indeed, there are no more than $n^{2 n }$ choices for $\bold{s}$ which satisfy $\sum_{i = 1}^{|\bold{s}|}\bold{s}_i = t$. Furthermore, for each such $\bold{s}$, we have that 
   \begin{align*}
       |\S_{j'}(\bold{s})| &\leq\prod_{i = 1}^{|\bold{s}|} \left( \frac{C'n^{2 - \frac{1}{\beta}}}{\bold{s}_i} \right)^{\left({\frac{\beta}{\beta - 1}}\right)\bold{s}_i}\cdot \exp(C'j \log(n)n^{2 - \alpha}).\\
       &= \exp\left(\sum_{i = 1}^{|\bold{s}|} \log\left( \frac{C'n^{2 - \frac{1}{\beta}}}{\bold{s}_i} \right) {\left({\frac{\beta}{\beta - 1}}\right)\bold{s}_i + C'j \log(n)n^{2 - \alpha}} \right)\\
       &\leq \exp\left(t \log\left( \frac{C'n^{2 - \frac{1}{\beta}}}{t} \right) {\left({\frac{\beta}{\beta - 1}}\right)\bold{s}_i + C'\log(n)^2n^{2 - \alpha}} + C'q_0^{1- \beta}n\right)
   \end{align*}
   where we apply \Cref{NetoMorris} using the fact that $\bold{s}_{j'- i} \leq (1 - \delta)^{i(\beta - 1) }(q_0)^{1 - \beta}n^2$ 
   This completes the proof of \Cref{appendixproposition}.

\end{proof}

\begin{proof}[Proof of Theorem~ \Cref{Lemma:General Random Turan for Large p}]
    Let $F$ be a graph and such that $F$ $(q_0(n)n^2, \gamma, \tau)$-balanced and for $n$ and $q \geq q_0(n)$, we have $\tau(n, q) = C \max\{n^{2 - \alpha}, q^{1 - \beta}n\}$. 

    Fix $q \geq q_0(n)$ and $0 <p \leq 1$ such that $pqn^2 \rightarrow \infty$ as $n \rightarrow \infty$. 

    By Lemma~\ref{appendixproposition}, there is a set $\c{S}$ and functions $f: \c{S} \rightarrow \c{P} (E(K_n))$and $g: \Forb(n, F) \rightarrow \c{S}$ with the following properties: 
    \begin{itemize}
    \item[(i)] Each element $\bold{S}\in \c{S}$ is a sequence of edge-disjoint subgraphs of $K_n$ such that $f(\bold{S})$ is edge-disjoint from each of the graphs $\bold{S}_i$.
    \item[(ii)] For each $H\in \Forb(n,F)$, we have
    \[\bigcup_i g(H)_i\sub H \sub f(g(H))\cup \bigcup_i g(H)_i.\]
    That is, every graph in the sequence $g(H)$ is a subgraph of $H$ and $H$ is contained in the union of these subgraphs and $f(g(H))$.
    \item[(iii)] We have $e(f(\bold{S}))\le q n^2$ for all $\bold{S}\in \c{S}$
    \item[(iv)] For every integer $t$, the number of $\bold{S}\in \c{S}$ with $\sum_i e(\bold{S}_i)=t$ is at most 
    $$\left( \frac{C'n^{2 - \frac{1}{\beta}}}{t
    }\right)^{\left({\frac{\beta}{\beta - 1}}\right)t} \cdot  \exp(C'q^{1 - \beta}n+ C'\log(n)^2n^{ 2- \alpha} )$$
\end{itemize}

Let $m =  12e^2 C'\max\{pqn^2, n^{2 - \alpha}(\log(n))^{C'}, q^{1 - \beta}n, p^{1 - \frac{1}{\beta}} n^{2 - \frac{1}{\beta}}\}$. For each $\bold{S} \in \S$, let $X_\bold{S}$ be the event that $|G(n, p) \cap f(\bold{S})| \geq m - \sum_i e(\bold{S_i})$. 

Then, 
\begin{align*}
    \mathbb{P}(\ex(F, G(n, p)) \geq m) &\leq \sum_{\bold{S} \in S} \mathbb{P}(\bold{S} \subseteq G(n, p)) \cdot \mathbb{P}(X_\bold{S}) \\ &\leq 
    \sum_{\bold{S} \in S} p^{\sum_i e(\bold{S_i})} \cdot \binom{q n^2}{m - \sum_i e(\bold{S_i}) }p^{m - \sum_i e(\bold{S_i}) }\\
    &\leq  \sum_{ t = 1}^{q^{1 - \beta} n} \left( \frac{C'n^{2 - \frac{1}{\beta} } p^{\frac{\beta - 1}{\beta}}}{t }\right)^{\left({\frac{\beta}{\beta - 1}}\right)t} \cdot  \exp(C'q^{1 - \beta}n+ C'\log(n)^2n^{ 2- \alpha} ) \cdot \left( \frac{eq p n^2}{ m - t} \right)^{ m - t}\\
    &\leq  \sum_{ t = 1}^{q^{1 - \beta} n} \left( \frac{C'n^{2 - \frac{1}{\beta} } p^{\frac{\beta - 1}{\beta}}}{t }\right)^{\left({\frac{\beta}{\beta - 1}}\right)t} \cdot  \exp(C'q^{1 - \beta}n+ C'\log(n)^2n^{ 2- \alpha} ) \cdot \left( \frac{eq p n^2}{ m - t} \right)^{ m - t}\\
 &\leq  \sum_{ t = 1}^{C'q^{1 - \beta} n} \exp(t\log(m / 12 e^2 t) + m/6 - m/2) \\
 &\leq \exp( - m / 4 + \log(m))\\
 &\leq \exp( -m / 8)
\end{align*}
Since $m \rightarrow \infty$ as $n \rightarrow \infty$, we have that  $\mathbb{P}(\ex(F, G(n, p)) \geq m) \rightarrow 0$ as $n \rightarrow \infty$, as desired.

\end{proof}

\end{document}